\documentclass[11pt,reqno]{amsart}
\usepackage[letterpaper, margin=1in]{geometry}
\usepackage{amsmath,amscd,amssymb}
\usepackage{url}
\usepackage{graphicx}

\usepackage{color}
\usepackage{tikz-cd}
\usepackage{enumerate}
\usepackage[pagebackref,colorlinks,citecolor=blue,linkcolor=magenta]{hyperref}
\usepackage[linesnumbered,ruled]{algorithm2e}
\RequirePackage{amsthm,amsmath,amsfonts,amssymb}
\usepackage[utf8]{inputenc}
\usepackage{chngcntr}
\usepackage{caption}
\usepackage{subcaption}
\usepackage{float}
\usepackage{placeins}
\usepackage{xcolor}
\usepackage{tikz}
\usetikzlibrary{decorations.pathreplacing,calligraphy}
\usetikzlibrary{arrows}
\usepackage{nicematrix}
\usepackage{bbm}
\usepackage[normalem]{ulem}

\usepackage{enumitem}

\usepackage{braket, mathtools}
\usepackage{cleveref}
\usepackage{parskip}
\usepackage[nocompress, noadjust]{cite}

\def\acts{\ensuremath{  \rotatebox[origin=c]{180}{$\circlearrowright$}}}

\makeatletter
\newtheorem*{rep@theorem}{\rep@title}
\newcommand{\newreptheorem}[2]{%
\newenvironment{rep#1}[1]{%
 \def\rep@title{#2 \ref{##1}}%
 \begin{rep@theorem}}%
 {\end{rep@theorem}}}
\makeatother

\newreptheorem{theorem}{Theorem}
\newtheorem{theorem}{Theorem}
\newtheorem*{theorem*}{Theorem}
\numberwithin{theorem}{section}
\newtheorem{proposition}[theorem]{Proposition}

\newtheorem{lemma}[theorem]{Lemma}

\theoremstyle{definition}
\newtheorem{definition}[theorem]{Definition}

\newtheorem*{definition*}{Definition}

\newtheorem{remark}[theorem]{Remark}

\newtheorem{example}[theorem]{Example}

\newtheorem*{question*}{Question}

\newcommand{\cc}{\mathbb{C}}

\newcommand{\ind}{\mathbbm{1}}

\newcommand{\cl}{\mathcal{L}}

\newcommand{\cm}{\mathcal{M}}

\newcommand{\cg}{\mathcal{G}}

\newcommand{\ch}{\mathcal{H}}

\newcommand{\from}{\leftarrow}

\newcommand{\fvv}{v}
\newcommand{\cimv}{c}
\newcommand{\simv}{s}

\newcommand{\boxedstar}{%
  {\setlength\fboxsep{1pt}%    % small padding inside box
   \setlength\fboxrule{0.8pt}% % thick border
   \fbox{$\star$}}%
}

\tikzstyle{wB}=[circle, draw, fill=black, inner sep=0pt, minimum width=4.5pt]
\tikzstyle{wR}=[circle, draw, fill=red, inner sep=0pt, minimum width=4.5pt]
\tikzstyle{wW}=[circle, draw, fill=white, inner sep=0pt, minimum width=4.5pt]

\DeclareMathOperator{\var}{Var}

\DeclareMathOperator{\pa}{pa}
\DeclareMathOperator{\fa}{fa}

\DeclareMathOperator{\chick}{\mathcal{C}}

\DeclareMathOperator{\diag}{diag}

\begin{document}

\title[Imsets for Cyclic Graphs]{Characteristic imsets for cyclic linear causal models and the Chickering ideal}

\date{June 16, 2025}

\author{Joseph Johnson}
\address{Institutionen f\"or Matematik, KTH, SE-100 44 Stockholm, Sweden}
\email{josjohn@kth.se, joejohnsondoesnumbers@gmail.com}

\author{Pardis Semnani}
\address{Department of Mathematics, University of British Columbia, 1984 Mathematics Road, Vancouver, BC, Canada V6T 1Z2}
\email{psemnani@math.ubc.ca}

\begin{abstract}
Two directed graphs are called covariance equivalent if they induce the same set of covariance matrices, up to a Lebesgue measure zero set, on the random variables of their associated linear structural equation models.  
For acyclic graphs, covariance equivalence is characterized both structurally, via essential graphs and characteristic imsets, and transformationally, through sequences of covered edge flips. However, when cycles are allowed, only a transformational characterization of covariance equivalence has been discovered. We consider a linear map whose fibers correspond to the sets of graphs with identical characteristic imset vectors, and study the toric ideal associated to its integer matrix. Using properties of this ideal we show that directed graphs with the same characteristic imset vectors are covariance equivalent. In applications, imsets form a smaller search space for solving causal discovery via greedy search.
\end{abstract}

\keywords{family variable vector, characteristic imset vector, structural equation models, covariance equivalence, toric ideals}
\subjclass[2020]{62H22, 13F65, 62E10, 13P25, 62D20}

\maketitle
\thispagestyle{empty}

\section{Introduction}
Learning the causal structure among different random variables of interest is a major challenge in various disciplines such as computational biology, social sciences, economics, and epidemiology~\cite{Friedman2000, pearlcausalitymodels, Robins2000, SpirtesGlymourClark1993}. 
While controlled trials may be used to discover these causal structures, these trials are not possible in all settings. For instance, randomly assigning patients to potentially harmful treatments poses ethical challenges in healthcare. Therefore, the problem of causal structure learning merely through observational data (and not interventional data) has received growing attention over the past decades.

One way to encode the causal structure among random variables $X_1,\ldots,X_n$ is through a 
\emph{structural equation model (SEM)}, which postulates that every variable is a noisy function of its direct causes,~i.e.,
\begin{align*}
    X_i = f_i(X_{\pa(i)}, \varepsilon_i) \qquad i\in [n], 
\end{align*}
for functions $f_1,\ldots,f_n$, subsets $\pa(i) \subseteq [n]\setminus \{i\}$ for $i\in [n]$, and jointly independent random noises $\varepsilon_1,\ldots,\varepsilon_n$. An SEM is naturally associated with a directed graph, where each node corresponds to a variable $X_i$, and each directed edge $i\to j$ indicates that $i\in \pa(j)$. 
A typical assumption in causal learning algorithms is that this directed graph is acyclic, i.e. it does not entail any directed cycles, also known as feedback loops. Within this framework many algorithms have been developed to learn the causal graph from observed data.
These methods can often be categorized as score-based~\cite{Heckerman1994LearningBN, meek1997graphical, Chickering2002OptimalSI}, constraint-based~\cite{SpirtesGlymourClark1993,Pearl2009}, or hybrid~\cite{tsamardinos2006max,Teyssier2005OrderingBasedSA, Raskutti2018, Solus2021, Bernstein2020}.

Although extensively assumed in algorithms, the acyclicity of the causal graph is not always realistic. Feedback loops are indeed present in the equilibrium state of many systems; see~\cite{Haavelmo1943, Goldberger1972, Mason1953, Mason1956} for examples of feedback loops in models related to economy, social sciences, and electrical engineering. Such examples have motivated the development of constraint-based~\cite{richardson2013discovery} and hybrid~\cite{SemnaniRobeva+2025} algorithms that learn the causal graph even in the presence of feedback loops. These algorithms use the conditional independence statements inferred from the data to learn as much about the graph as possible. 

Under various assumptions, such as acyclicity of the graph~\cite{pearlcausalitymodels}, linearity of the functions $f_i$ in the SEM~\cite{forré2017markov}, or discreteness of the variables~\cite{forré2017markov}, distributions of the random vectors $X = (X_1, \ldots, X_n)$ in an SEM satisfy the \emph{global Markov property} with respect to its associated graph, i.e., the distributions satisfy conditional independence statements induced by the graph via its $d$-separation statements. Consequently, algorithms which only rely on the conditional independence statements inferred from observational data can only recover the causal graph up to the class of graphs which induce the same $d$-separation statements, called its \emph{Markov equivalence class}.

Different graphs can encode the same set of distributions in their associated SEMs, a phenomenon called \emph{model equivalence}. For linear SEMs with independent Gaussian noise that are associated to directed acyclic graphs (DAGs), Markov equivalence and model equivalence are the same relation; see~\cite{sullivant2018algstat}.
In the presence of cycles in the graph or non-Gaussian noise in the SEM, however, this is no longer the case: two graphs may differ in the sets of distributions they induce, yet still encode the same set of $d$-separations and thus be Markov equivalent~\cite{Lacerda2008}. 
In the cyclic setting, Markov equivalent graphs may even give rise to different sets of covariance matrices in the linear SEMs associated with them. Therefore, relying solely on the observed conditional independencies/dependencies limits the extent to which we can identify possibly cyclic causal graphs from the observed data.

We call two directed graphs \emph{covariance equivalent} if they give rise to the same set of covariance matrices in their associated linear SEMs, up to a Lebesgue measure zero set. 
A transformational characterization of covariance equivalence is given in~\cite{pmlr-v119-ghassami20a}, and~\cite{pmlr-v124-amendola20a} contains a niche  sufficient condition for two possibly cyclic mixed graphs to be covariance equivalent. In the more restricted setting of fully observed cyclic models, we prove a more widely applicable criterion for deciding if two graphs are covariance equivalent based on vector representatives.

The purpose of the present article is to explore the relationship between the \emph{characteristic imset vector} of a directed (possibly cyclic) graph and its covariance equivalence class. In~\cite{studeny2010characteristic},  Studený, Hemmecke and Lindner associate with each DAG $\cg$ on $n$ vertices its \emph{characteristic imset vector} $\cimv_\cg\in \{0,1\}^{2^n-1}$, and prove that this vector characterizes the covariance equivalence class of $\cg$; see \Cref{def:cim} and \Cref{thm: cimv and markov equivalence}. Studený shows that classical score functions for causal discovery over DAGs can be expressed as linear functionals over the convex hull of all characteristic imset vectors corresponding to DAGs~\cite{studeny2005book}, motivating linear programming approaches~\cite{hollering2024hyperplanerepresentationsinterventionalcharacteristic} for causal discovery in the acyclic setting. In this paper, we introduce a natural extension of characteristic imset vectors to accommodate graphs with directed cycles. Unlike the acyclic case, these vectors do not need to be $0/1$-vectors. Instead, the characteristic imset vector of a graph is a $0/1$-vector if and only if the graph contains no 2-cycles. 

In this article we study the extent to which characteristic imsets  refine the covariance equivalence relation in the cyclic setting. While in this setting we may not translate causal discovery via decomposable score functions, such as the BIC, into a linear program,
performing greedy search algorithms for causal discovery over the set of characteristic imset vectors is still expected to yield faster algorithms since the search space is smaller. However, we leave such algorithms to~future~work.

In algebraic statistics, vectors that represent equivalence classes are closely related to the study of discrete exponential families and toric ideals~\cite{sullivant2018algstat}. In this paradigm, sufficient statistics form unique representatives for fibers of a linear map, and Markov bases form moves that connect the elements of a fiber. We build on this connection by considering a linear map $\varphi_n: \mathbb{R}^{n\cdot 2^{n-1}} \to \mathbb{R}^{2^n-1}$  from the space of \emph{family variable vectors} \cite{cussens2013maximum, jaakkola2010learning} to the space of characteristic imset vectors such that fibers of this map correspond to the collections of directed graphs with the same characteristic imset vector. See \cite{cussens2017scoreequiv} for some prior work regarding this linear map. We analyze the algebraic structure of the toric ideal associated to the integer matrix of $\varphi_n$, which we refer to as the \emph{Chickering ideal}, and exploit the  properties of the Chickering ideal to obtain our main result:
\begin{reptheorem}{thm: main}
    Directed graphs with the same characteristic imset vector are covariance equivalent.
\end{reptheorem}

This paper is organized as follows. In Section~\ref{sec:background} we cover the relevant background on characterizations of equivalence classes of graphs, characteristic imset vectors, and family variable vectors. In \Cref{sec:chickeringIdeal}, we define the Chickering ideal, discuss the difficulties in computing a generating set, and relate the Chickering ideal to a radical binomial ideal through elimination. In \Cref{sec:matrixMoves}, we use the properties of the Chickering ideal to prove our main result. We conclude the paper with \Cref{sec:futureDirections}, which is a discussion of possible future directions for this work.

\section{Background}
\label{sec:background}

In this section we cover the necessary background on linear structural equation models, imsets, and family variable vectors. Throughout this section we fix a positive integer $n$ and consider (possibly cyclic) loopless directed graphs $\cg = \left(V(\cg),E(\cg)\right)$ on the vertex set $[n]$. For $b,c \in [n]$, we often suppress set notation for singletons, writing $b$ instead of $\{b\}$ in proofs. Additionally, in examples we often write $bc$ in place of $\{b,c\}$.

\subsection{Model and Covariance Equivalence} In this section, we introduce two notions of equivalence among directed (possibly cyclic) graphs. We begin by the following definition.
\begin{definition}
Let $v \in [n]$. We define the set of \emph{parents} of $v$ in $\cg$ to be $\pa_\cg(v) \coloneqq \{u \in [n]: u \to v \in E(\cg)\}$, and the \emph{family} of $v$ in $\cg$ to be $\fa_\cg(v) \coloneqq \pa_\cg(v) \cup \{v\}$.
\end{definition}

Linear structural equation models relate the coordinates of a random vector via linear equations whose sparsity is determined by $\cg$.

\begin{definition}
    For a directed graph $\cg$, we denote by $\cm(\cg)$ the set of all random vectors $(X_1,\ldots,X_n)$ satisfying a \emph{linear structural equation model (SEM)} associated with $\cg$, 
    \begin{align}\label{eq: linear SEM}
        X_i = \sum_{j \in \pa_\cg(i)} \lambda_{ij} X_j + \varepsilon_i \quad \text{for all } i \in [n],
    \end{align}
    where for all $j\to i\in E(\cg)$, $\lambda_{ij} \in \mathbb{R}$,  and $\varepsilon_1,\ldots,\varepsilon_n$ are jointly independent random variables. Additionally, for some fixed values of  $\lambda_{ij}$ and $\varepsilon_i$, if \Cref{eq: linear SEM} has two distinct solutions, we exclude both solutions from~$\mathcal{M}(\cg)$.
\end{definition}

We are now ready to define \emph{model equivalence} and \emph{covariance equivalence} classes of directed graphs.
\begin{definition} \label{def: distributional equivalence}
We call two directed graphs $\cg$ and $\ch$ \emph{covariance equivalent} if the set of precision matrices (i.e. inverse covariance matrices) of the random vectors in $\cm(\cg)$ and $\cm(\ch)$ have the same Euclidean closures. 
Furthermore, we call $\cg$ and $\ch$ \emph{model equivalent} if the set of distributions of the random vectors in $\cm(\cg)$ and $\cm(\ch)$ are the same.  
\end{definition}

\begin{remark}

Model equivalence is, in general, a stronger notion than covariance equivalence. However, in the case where the noise $(\varepsilon_1, \ldots, \varepsilon_n)$ is Gaussian with mean zero, covariance equivalence and model equivalence are the same relation on directed graphs.
\end{remark}

For the random vector $X$ satisfying the linear SEM in~\Cref{eq: linear SEM}, assume $\var(\varepsilon_i) = \omega_i$ for all $i\in [n]$, and define the matrix $\Lambda \in \mathbb{R}^{n\times n}$ with 
\begin{align*}
    \Lambda_{ij} = \begin{cases}
        \lambda_{ij} & \text{if } j \to i \in E(\cg), \\
        0 & \text{otherwise.}
    \end{cases}
\end{align*}
Then if $I - \Lambda$ is invertible, the covariance matrix of $X$ is equal to 
$
    (I - \Lambda)^{-1} \Omega (I - \Lambda)^{-T} 
$,
where $\Omega = \diag(\omega_1,\ldots,\omega_n)$. Therefore, the set of precision matrices of random vectors $X\in \cm(\cg)$ can be parameterized by the map
\begin{align*}
    (\Lambda, \Omega) \mapsto (I - \Lambda)^{T} \Omega^{-1} (I-\Lambda)
\end{align*}
for all positive-definite diagonal matrices $\Omega \in \mathbb{R}^{n \times n}$ and for all matrices $\Lambda \in \mathbb{R}^{n\times n}$ such that $\Lambda_{ij} = 0$ if $j \to i  \not\in E(\cg)$ and $I-\Lambda$ is invertible. We have the following reformulation of this parameterization in~\cite{pmlr-v119-ghassami20a}.

\begin{proposition}
\label{prop:sparseCholeskyFactor}
The set of precision matrices of the random vectors $X\in \cm(\cg)$ is parameterized by the map
\begin{align*}
    Q \mapsto QQ^T,
\end{align*}
where $Q$ is any full rank matrix in $\mathbb{R}^{n \times n}$ such that $Q_{ij} = 0$ when $i\neq j\in [n]$ and $i \to j \not\in E(\cg)$, and $Q_{ii} \neq 0$ for all $i\in [n]$. In particular, the sparsity of the $j$th column of $Q$ is given by $\fa_\cg(j)$.
\end{proposition}

\begin{remark}
\label{rem:unitaryTransformations}
For two full-rank matrices $Q_1, Q_2 \in \mathbb{R}^{n \times n}$, we have $Q_1 Q_1^T = Q_2 Q_2^T$ if and only if there exists some orthogonal transformation $U \in \mathbb{R}^{n \times n}$ with $Q_2 = Q_1 U$. In this case, 
\begin{align*}
    Q_2 Q_2^T = (Q_1 U)(Q_1 U)^T = Q_1 UU^TQ_1^T = Q_1Q_1^T.
\end{align*}
Hence orthogonal transformations  can change the sparsity of the factor $Q_1$ and show that its corresponding precision matrix belongs to two different models.

\end{remark}

\begin{example}
\label{ex:modelEquivalenceOfCycles}
Consider the graphs $\cg$ and $\ch$ in Figure~\ref{fig:equivalentWithDifferentSkeletons}. Let $Q_\cg$ have the sparsity associated to $\cg$, as described in Proposition~\ref{prop:sparseCholeskyFactor}. After multiplying $Q_\cg$ by a permutation matrix, we obtain a matrix whose sparsity corresponds to $\ch$.
\begin{align*}
\bordermatrix{ 
& {\emptyset \to 1} & {4 \to 2} & {12 \to 3} & {3 \to 4} \cr
& \star &     0 & \star &     0 \cr
&     0 & \star & \star &     0 \cr
&     0 &     0 & \star & \star \cr
&     0 & \star &     0 & \star \cr
} \begin{pmatrix}
1 & 0 & 0 & 0 \\
0 & 0 & 0 & 1 \\
0 & 1 & 0 & 0 \\
0 & 0 & 1 & 0
\end{pmatrix} = \bordermatrix{ 
& {\emptyset \to 1} & {13 \to 2} & {4 \to 3} & {2 \to 4} \cr
& \star & \star &     0 &     0 \cr
&     0 & \star &     0 & \star \cr
&     0 & \star & \star &     0 \cr
&     0 &     0 & \star & \star \cr
} 
\end{align*}
The matrix on the right hand side has the sparsity given by $\ch$, and so $\cg$ and $\ch$ are covariance equivalent.
\end{example}

\begin{figure}
    \centering
\begin{tikzpicture}
\begin{scope}[shift={(0,0)}]
\node at (-0.8,0) {$\cg$};
\node[wB] at (1,1) {};
\node[wB] at (0,0) {};
\node[wB] at (1,0) {};
\node[wB] at (0.5,-0.866) {};

\node at (1.3,1) {1};
\node at (-0.3,0) {2};
\node at (1.3,0) {3};
\node[below=2pt] at (0.5,-0.866) {4};

\draw[-Stealth] (1,1)--(1,0.1);
\draw[-Stealth] (1,0)--(0.55,-0.779);
\draw[-Stealth] (0.5,-0.866)--(0.05,-0.09);
\draw[-Stealth] (0,0)--(0.9,0);
\end{scope}

\draw [ultra thick, -Stealth] (1.75,0.067)--(3.25,0.067);

\begin{scope}[shift={(4,0)}]
\node at (1.8,0) {$\ch$};

\node[wB] at (1,1) {};
\node[wB] at (0,0) {};
\node[wB] at (1,0) {};
\node[wB] at (0.5,-0.866) {};

\node at (1.3,1) {1};
\node at (-0.3,0) {2};
\node at (1.3,0) {3};
\node[below=2pt] at (0.5,-0.866) {4};

\draw[-Stealth] (1,1)--(0.07,0.07);
\draw[-Stealth] (0,0)--(0.45,-0.779);
\draw[-Stealth] (0.5,-0.866)--(0.95,-0.09);
\draw[-Stealth] (1,0)--(0.1,0);
\end{scope}

\end{tikzpicture}
    \caption{Two graphs which are covariance equivalent, but have different skeletons. These graphs are related by reversing the cycle $2 \to 3 \to 4 \to 2$ in the sense of~\cite{pmlr-v119-ghassami20a}.}
    \label{fig:equivalentWithDifferentSkeletons}
\end{figure}
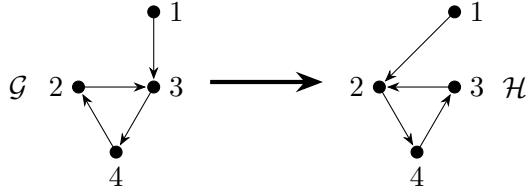

\Cref{rem:unitaryTransformations} gives rise to a \emph{transformational characterization} of covariance equivalence for possibly cyclic directed graphs~\cite{pmlr-v119-ghassami20a}, i.e. a collection of local transformations on graphs which relate the set of covariance equivalent graphs.  However, when restricted to directed acyclic graphs (DAGs), the notions of covariance equivalence and Markov equivalence coincide, and their transformational characterization, due to Chickering \cite{chickering1995transformational}, becomes much simpler.

\begin{definition}
Let $\cg$ be a directed graph. An edge $i \to j$ is \emph{covered} if $\pa_\cg(j) = \fa_\cg(i)$. A \emph{covered edge flip} is the transformation on a graph defined by reversing the orientation of a covered edge.
\end{definition}
\begin{theorem}[Chickering's Transformational Characterization, \cite{chickering1995transformational}]
Two DAGs $\cg$ and $\ch$ are covariance/Markov equivalent if and only if there exists a sequence $\cg = \cg_1, \cg_2,\dots, \cg_k = \ch$ such that for all $i\in [k-1]$, $\cg_i$ and $\cg_{i+1}$ are related by a covered edge flip.
\end{theorem}

Furthermore, in the case of DAGs, a structural characterization of covariance/Markov equivalence is well-known.

\begin{definition}
The \emph{skeleton} of a directed graph $\cg$ is the set of its underlying undirected edges, with multiplicity. A \emph{v-structure} of $\cg$ is an induced subgraph of $\cg$ of the form $i \to k \from j$.    
\end{definition}

\begin{theorem}[\cite{verma2022equivalence}]
\label{thm:vermaPearlCharacterization}
Let $\cg,\ch$ be two DAGs on the vertex set $[n]$. Then $\cg$ and $\ch$ are covariance/Markov equivalent if and only if they have the same skeleton and v-structures.
\end{theorem}

In score-based causal discovery methods, one searches over a subset of the set of all directed graphs for a graph in the optimally scored equivalence class. 
In this search, one may encounter multiple graphs in the same equivalence class, and so the score for an equivalence class may be computed multiple times. Consequently, it is useful to study this equivalence relation and find unique representatives of each class.
The characterization of covariance equivalence in Theorem~\ref{thm:vermaPearlCharacterization} gives rise to a unique combinatorial representative of a covariance equivalence class of DAGs, called the \emph{essential graph}.

\begin{definition}
\label{def:essentialGraph}
The \emph{essential graph} of a DAG $\cg$ is the graph obtained from $\cg$ by removing the orientation of edges which are flipped in at least one Markov equivalent DAG to $\cg$.
\end{definition}

\begin{theorem}[\cite{andersson1997markov}]
Two DAGs $\cg$ and $\ch$ are covariance/Markov equivalent if and only if they have the same essential graph.    
\end{theorem}

\begin{example}
In \Cref{fig: Markov equivalent DAGs}, three covariance equivalent DAGs, related via covered edge flips, are illustrated. Note that these graphs have the same skeleton and v-structures, and any other configuration of edge orientations would destroy a v-structure, add a v-structure, or  create a directed cycle in the graph. Hence, the three DAGs in \Cref{fig: Markov equivalent DAGs} are the only DAGs in this covariance equivalence class, and its associated essential graph is the graph shown in \Cref{fig: essential graph}.
\end{example}

\begin{figure}[t!]
\centering
\begin{subfigure}[t]{0.7\textwidth}
\centering
\begin{tikzpicture}[>=Stealth, shorten >=0pt, shorten <=0pt]

% Define spacing
\def\dx{1.5}
\def\sep{2.4} % separation between graph blocks

% Graph 1
\begin{scope}[shift={(0,0)}]
\node[wB] (X1) at (0,\dx) {};
\node[wB] (X2) at (\dx,\dx) {};
\node[wB] (X3) at (\dx,0) {};
\node[wB] (X4) at (0,0) {};
\node at (-0.25,\dx) {1};
\node at (\dx+0.25,\dx) {2};
\node at (\dx+0.25,0) {3};
\node at (-0.25,0) {4};

\draw[->] (X2) -- (X1);
\draw[->] (X4) -- (X1);
\draw[->] (X3) -- (X1);
\draw[->] (X3) -- (X2);
\draw[->] (X4) -- (X3);
\end{scope}

% Arrow between Graph 1 and 2
\draw [ultra thick, -Stealth] 
  (\dx+0.5, \dx/2) --++ (\sep - 1, 0);

% Graph 2
\begin{scope}[shift={(\sep + \dx,0)}]
\node[wB] (X1) at (0,\dx) {};
\node[wB] (X2) at (\dx,\dx) {};
\node[wB] (X3) at (\dx,0) {};
\node[wB] (X4) at (0,0) {};
\node at (-0.25,\dx) {1};
\node at (\dx+0.25,\dx) {2};
\node at (\dx+0.25,0) {3};
\node at (-0.25,0) {4};

\draw[->] (X2) -- (X1);
\draw[->] (X4) -- (X1);
\draw[->] (X3) -- (X1);
\draw[->] (X3) -- (X2);
\draw[->] (X3) -- (X4);
\end{scope}

% Arrow between Graph 2 and 3
\draw [ultra thick, -Stealth] 
  ({2*\sep + 3*\dx/2 - 1}, \dx/2) --++ (\sep - 1, 0);

% Graph 3
\begin{scope}[shift={(2*\sep + 2*\dx,0)}]
\node[wB] (X1) at (0,\dx) {};
\node[wB] (X2) at (\dx,\dx) {};
\node[wB] (X3) at (\dx,0) {};
\node[wB] (X4) at (0,0) {};
\node at (-0.25,\dx) {1};
\node at (\dx+0.25,\dx) {2};
\node at (\dx+0.25,0) {3};
\node at (-0.25,0) {4};

\draw[->] (X2) -- (X1);
\draw[->] (X4) -- (X1);
\draw[->] (X3) -- (X1);
\draw[->] (X2) -- (X3);
\draw[->] (X3) -- (X4);
\end{scope}

\end{tikzpicture}
\caption{}
\label{fig: Markov equivalent DAGs}
\end{subfigure}%
    ~ 
\begin{subfigure}[t]{0.25\textwidth}
\centering
\begin{tikzpicture}[>=Stealth]

% Define spacing
\def\dx{1.5}
\def\sep{2.4} % separation between graph blocks

% Essential graph
\node[wB] (X1) at (0,\dx) {};
\node[wB] (X2) at (\dx,\dx) {};
\node[wB] (X3) at (\dx,0) {};
\node[wB] (X4) at (0,0) {};
\node at (-0.25,\dx) {1};
\node at (\dx+0.25,\dx) {2};
\node at (\dx+0.25,0) {3};
\node at (-0.25,0) {4};

\draw[->] (X2) -- (X1);
\draw[->] (X4) -- (X1);
\draw[->] (X3) -- (X1);
\draw[-] (X2) -- (X3);
\draw[-] (X4) -- (X3);

\end{tikzpicture}
\caption{}
\label{fig: essential graph}
\end{subfigure}
\caption{(A) Three covariance/Markov equivalent DAGs, which are related by flipping the covered edges $4 \to 3$ and then $3 \to 2$, (B) along with their corresponding  essential graph. }
\label{fig: equivalence in DAGs}
\end{figure}
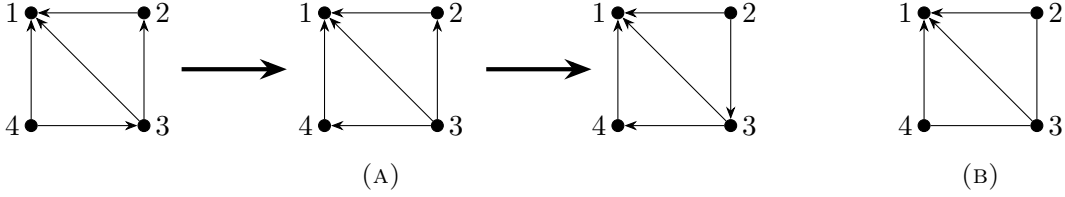

When directed cycles are allowed in the graphs, the theory is less well understood. As we saw in Example~\ref{ex:modelEquivalenceOfCycles}, graphs with different skeletons can be covariance equivalent. Therefore, it is difficult to give a structural characterization of this equivalence relation. In this paper, we present a sufficient structural condition for two graphs to be covariance equivalent which applies more broadly than the condition presented in~\cite{pmlr-v124-amendola20a}.

\subsection{Imsets and the Family Variable Vectors} Another representative of covariance equivalence for acyclic graphs arises in the form of characteristic imset vectors, which are 0/1 vectors that record structural features of graphs. In this section we review the literature on these objects.

\begin{definition}[Characteristic imset vectors]
\label{def:cim}
    Let $\mathcal{G} = ([n], E(\mathcal{G}))$ be a directed (possibly cyclic) graph. The \emph{characteristic imset vector} $c_\cg \in \mathbb{R}^{2^n-1}$ associated with $\cg$ is the vector indexed by nonempty sets $A\subseteq [n]$ defined by
    \begin{align*}
        \cimv_\mathcal{G}(A) \coloneqq
            \# \{a \in A : A\setminus \{a\} \subseteq \pa_\mathcal{G}(a) \},
    \end{align*}
    where for a set $S$, $\#S$ denotes the size of $S$.
\end{definition}

Including coordinates indexed by singleton sets is not necessary for distinguishing between graphs since for each $b \in [n]$ we have $c_\cg(b) = 1$. 
However, we will see in later sections that including these coordinates is relevant for our algebraic computations; see Remark~\ref{rem:reasonForSingletons}. Characteristic imsets characterize the covariance/Markov equivalence classes when restricted to DAGs.

\begin{theorem}[\cite{studeny2010characteristic}]
\label{thm: cimv and markov equivalence}
Let $\cg,\ch$ be two DAGs on $[n]$. Then $\cg$ and $\ch$ are covariance/Markov equivalent if and only if $c_\cg = c_\ch$.
\end{theorem}

In \cite{cussens2013maximum,jaakkola2010learning,ijcai2017p708} a different vector is studied, which uniquely characterizes each directed graph.

\begin{definition}[Family variable vectors]
\label{def:fvp}
Let $\mathcal{G} = ([n], E(\mathcal{G}))$ be a directed (possibly cyclic) graph. The \emph{family variable vector} $\fvv_\mathcal{G} \in \mathbb{R}^{n \cdot 2^{n-1}}$ associated with $\mathcal{G}$ is the vector indexed by families $A \to b$ with
$b \in [n]$ and $A \subseteq [n] \setminus \{b\}$, defined by
\begin{align*}
    \fvv_{\mathcal{G}}(A\to b) \coloneqq 
    \begin{cases}
        1 & \text{if } \pa_\mathcal{G}(b) = A, \\
        0 & \text{otherwise}.
    \end{cases}
\end{align*}
\end{definition}

Family variable vectors are in bijection with directed graphs since each vector records the families contained in the graph.
We have the following map between the space of family variable vectors and the space of characteristic imset vectors.

\begin{definition}
\label{def:mapBetweenPolytopes}
Define $\varphi_n:\mathbb{R}^{n \cdot 2^{n-1}} \to \mathbb{R}^{2^n-1}$ to be the linear map given by
\begin{align*}
(\varphi_n(u))(S) = \sum\limits_{b \in S} \sum\limits_{A: [n]\setminus \{b\} \supseteq A \supseteq S \setminus \{b\}} u(A \to b) \quad \text{for all nonempty set $S \subseteq [n]$}.
\end{align*}
\end{definition}

\begin{lemma}[Lemma~5, \cite{cussens2017scoreequiv}]
\label{lemma:quotientMap}
For all directed graphs $\cg$ on $[n]$, $\varphi_n(\fvv_\cg) = c_\cg$. Moreover, the kernel of $\varphi_n$ is the vector space generated by all vectors of the form $e_{A \to b} + e_{A \cup b \to c} - e_{A \to c} - e_{A \cup c \to b}$ where $b,c \in [n]$ are distinct,  $A \subseteq [n] \setminus \{b,c\}$, and $e_{A \to b}$ denotes the standard basis vector associated to the coordinate indexed by the family $A \to b$.
\end{lemma}

Note that the generators in Lemma~\ref{lemma:quotientMap} correspond to covered edge flips. We refer readers to Figure~\ref{fig:coveredEdgeFlip}; the left graph corresponds to the standard basis vectors with coefficient $1$ and the right graph corresponds to standard basis vectors with coefficient $-1$. We have the following result regarding the range of $\varphi_n$.

\begin{proposition}
\label{prop:upperTriangular}
Let $M \in \mathbb{Z}^{(2^n-1) \times (n \cdot 2^{n-1})}$ denote the matrix of the linear map $\varphi_n$. The columns of $M$ indexed by $A \to b$, where $\max(A) < b$, are linearly independent.
\end{proposition}

\begin{proof}
Order the rows of $M$, indexed by $\emptyset \neq S \subseteq [n]$, in the order of a linear extension of the Boolean lattice, and order the columns of $M$, indexed by $A \to b$, in the same order with respect to the sets $A \cup \{b\}$. Note that for any $u\in \mathbb{R}^{n \cdot 2^{n-1}}$ and  any $\emptyset \neq S \subseteq [n]$,  
\begin{align*}
(\varphi_n(u))(S) = \sum\limits_{b \in S} \sum\limits_{ A \supseteq S \setminus b} u(A \to b).
\end{align*}
Therefore, given the specific ordering of the rows and columns of $M$, the submatrix of $M$ containing columns indexed by $A\to b$ with $\max(A)<b$, is upper triangular with 1s on the diagonal. In particular, this submatrix has full column rank.
\end{proof}

We conclude this section with a brief discussion of standard imsets.

\begin{definition}[Standard imset vectors]
\label{def:standardimsets}
  Let $\cg = ([n], E(\cg))$ be a directed (possibly cyclic) graph. For each set $A \subseteq [n]$ define a vector $\delta_A \in \mathbb{R}^{2^n}$ whose coordinates are indexed by sets $B \subseteq [n]$ such that
  \begin{align*}
      \delta_A(B) \coloneqq \begin{cases}
          1 & \text{if } B = A, \\
          0 & \text{otherwise}.
      \end{cases}
  \end{align*}

  The \emph{standard imset vector} $\simv_\mathcal{G} \in \mathbb{R}^{2^n}$ associated with $\cg$ is the vector indexed by sets $A\subseteq [n]$, defined by 
  \begin{align*}
      \simv_\cg \coloneqq \sum_{i \in [n]} (\delta_{\fa_\cg(i)} - \delta_{\pa_\cg(i)}).
  \end{align*}
The definition of the standard imset yields a linear map $\psi_n:\mathbb{R}^{n \cdot 2^{n-1}} \to \mathbb{R}^{2^n}$ where $\psi_n(e_{A \to b}) = \delta_{A \cup \{b\}} - \delta_A$, so that $\psi_n(\fvv_\cg) = s_\cg$ for all $\cg$.
\end{definition}

Our definition differs slightly from the usual convention; see \cite{studeny2005book}. Note the sparsity in standard imsets -- only linearly many coordinates in $n$ are nonzero. Additionally, the form of the standard imset reflects the recursive factorization property for distributions that are Markov to a DAG after rewriting conditional densities as quotients of marginal densities. Together, these two properties make standard imsets more practical in applied settings than characteristic imsets. 

In the acyclic setting, it is easy to see that standard imsets form unique representatives for Markov equivalence classes since they mimic the recursive factorization written in terms of marginal densities. However, in the cyclic setting the recursive factorization yields a product of kernels, not densities, and the relationship is less clear.

Standard imsets and characteristic imsets are related via the following linear isomorphism. 

\begin{proposition}[\cite{studeny2010characteristic}]
\label{prop: standard imset}
    For all directed graphs $\cg$ and all nonempty subsets $A\subseteq [n]$, 
    \begin{align*}
        \cimv_\cg (A) = \sum_{B : A \subseteq B \subseteq [n]} \simv_\cg(B).
    \end{align*}
    Hence, using the M\"{o}bius Inversion formula, for all nonempty subsets $B\subseteq [n]$,
    \begin{align*}
        \simv_\cg (B) = \sum_{A : B \subseteq A \subseteq [n]} (-1)^{|A \setminus B|} \cimv_\cg(A) 
    \end{align*}
    and the coordinate associated to the empty set is obtained from the equation $\sum\limits_{A \subseteq [n]} s_\cg(A) = 0$.
\end{proposition}

We provide a proof for completeness.

\begin{proof}
It suffices to prove the first equation in Proposition~\ref{prop: standard imset} since the second equation follows directly via M\"{o}bius inversion. For all nonempty subsets $A \subseteq [n]$,
\begin{align*}
\sum\limits_{B \supseteq A} s_\cg(B) = \sum\limits_{\substack{a \in [n] \\ \fa_\cg(a) \supseteq A}} 1 - \sum\limits_{\substack{a \in [n] \\ \pa_\cg(a) \supseteq A}} 1. 
\end{align*}
Clearly if $A \subseteq \pa_\cg(a)$, then $A \subseteq \fa_\cg(a)$, and so all of the terms in the second sum cancel with some term in the first. What remains is
\begin{align*} 
\sum\limits_{B \supseteq A} s_\cg(B) = \#\left\{ a \in A : \fa_\cg(a) \supseteq A \right\}
= c_\cg(A).
\end{align*}
\end{proof}

\begin{example}
Let $\cg$ and $\ch$ be the two directed graphs on $n=4$ vertices, illustrated in \Cref{fig: covariance equivalent with the same skeleton}. Then $\fvv_\cg(A\to b) = 1$ only if $A\to b\in \{4 \to 1, 1 \to 2, 12 \to 3, 23\to 4 \}$, and $\fvv_\ch(A\to b) = 1$ only if $A\to b \in \{23 \to 1, 34\to 2, 4 \to 3, 1\to 4\}$. These vectors encode exactly the families of each graph, and so $\fvv_\cg \neq \fvv_\ch$. However, their characteristic imset and standard imset vectors are the same:
\begin{align*}
&\cimv_\cg^T = \cimv_\ch^T = \bordermatrix{ 
& 1 & 2 & 3 & 4 & 12 & 13 & 14 & 23 & 24 & 34 & 123 & 134 & 124 & 234 & 1234 \cr 
& 1 & 1 & 1 & 1 & 1 & 1 & 1 & 1 & 1 & 1 & 1 & 0 & 0 & 1 & 0
},\\
&\simv_\cg^T = \simv_\ch ^T=
\bordermatrix{ 
& \emptyset & 1 & 2 & 3 & 4 & 12 & 13 & 14 & 23 & 24 & 34 & 123 & 134 & 124 & 234 & 1234 \cr 
&0 & -1 & 0 & 0 & -1 & 0 & 0 & 1 & -1 & 0 & 0 & 1 & 0 & 0 & 1 & 0
}.
\end{align*}

\begin{figure}
\centering
\begin{tikzpicture}[>=Stealth, shorten >=0pt, shorten <=0pt]

% Define spacing
\def\dx{1.5}
\def\sep{2.4} % separation between graph blocks

% Graph 1
\begin{scope}[shift={(0,0)}]
\node[wB] (X1) at (0,\dx) {};
\node[wB] (X2) at (\dx,\dx) {};
\node[wB] (X3) at (\dx,0) {};
\node[wB] (X4) at (0,0) {};
\node at (-0.25,\dx) {1};
\node at (\dx+0.25,\dx) {2};
\node at (\dx+0.25,0) {3};
\node at (-0.25,0) {4};

\draw[->] (X1) -- (X2);
\draw[->] (X2) -- (X3);
\draw[->] (X3) -- (X4);
\draw[->] (X4) -- (X1);
\draw[->] (X1) -- (X3);
\draw[->] (X2) -- (X4);
\end{scope}

% Arrow between Graph 1 and 2
\draw [ultra thick, -Stealth] 
  (\dx+0.5, \dx/2) --++ (\sep - 1, 0);

% Graph 2
\begin{scope}[shift={(\sep + \dx,0)}]
\node[wB] (X1) at (0,\dx) {};
\node[wB] (X2) at (\dx,\dx) {};
\node[wB] (X3) at (\dx,0) {};
\node[wB] (X4) at (0,0) {};
\node at (-0.25,\dx) {1};
\node at (\dx+0.25,\dx) {2};
\node at (\dx+0.25,0) {3};
\node at (-0.25,0) {4};

\draw[->, bend right=15] (X1) to (X2);
\draw[->, bend right=15] (X2) to (X1);

\draw[->, bend right=15] (X2) to (X3);
\draw[->, bend right=15] (X3) to (X2);

\draw[->] (X4) -- (X3);
\draw[->] (X1) -- (X4);
\end{scope}

% Arrow between Graph 2 and 3
\draw [ultra thick, -Stealth] 
  ({2*\sep + 3*\dx/2 - 1}, \dx/2) --++ (\sep - 1, 0);

% Graph 3
\begin{scope}[shift={(2*\sep + 2*\dx,0)}]
\node[wB] (X1) at (0,\dx) {};
\node[wB] (X2) at (\dx,\dx) {};
\node[wB] (X3) at (\dx,0) {};
\node[wB] (X4) at (0,0) {};
\node at (-0.25,\dx) {1};
\node at (\dx+0.25,\dx) {2};
\node at (\dx+0.25,0) {3};
\node at (-0.25,0) {4};

\draw[->, bend right=15] (X2) to (X3);
\draw[->, bend right=15] (X3) to (X2);

\draw[->] (X2) -- (X1);
\draw[->] (X1) -- (X4);
\draw[->] (X4) -- (X3);
\draw[->] (X3) -- (X1);
\end{scope}

% Arrow between Graph 3 and 4
\draw [ultra thick, -Stealth] 
  ({3*\sep + 5*\dx/2 - 1}, \dx/2) --++ (\sep - 1, 0);

% Graph 4
\begin{scope}[shift={(3*\sep + 3 * \dx,0)}]
\node[wB] (X1) at (0,\dx) {};
\node[wB] (X2) at (\dx,\dx) {};
\node[wB] (X3) at (\dx,0) {};
\node[wB] (X4) at (0,0) {};
\node at (-0.25,\dx) {1};
\node at (\dx+0.25,\dx) {2};
\node at (\dx+0.25,0) {3};
\node at (-0.25,0) {4};

\draw[->] (X2) -- (X1);
\draw[->] (X1) -- (X4);
\draw[->] (X4) -- (X3);
\draw[->] (X3) -- (X1);
\draw[->] (X4) -- (X2);
\draw[->] (X3) -- (X2);
\end{scope}
\end{tikzpicture}

    \caption{The leftmost graph $\cg$ and the rightmost graph $\ch$ are covariance equivalent. They are related by reversing the cycle $1 \to 2 \to 3 \to 4 \to 1$, then exchanging the parents of nodes $1,2$, and finally exchanging the parents of nodes $2,3$. These operations are done in the sense of~\cite{pmlr-v119-ghassami20a}, and result in the above graphs respectively.}
    \label{fig: covariance equivalent with the same skeleton}
\end{figure}
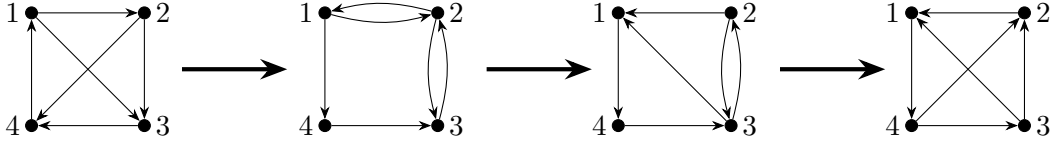

\end{example}

\section{The Chickering Ideal}
\label{sec:chickeringIdeal}

In this section we define a toric ideal we call the \emph{Chickering ideal} and demonstrate the difficulties of studying this ideal through its generators. By creating duplicates of variables, we form a binomial \emph{doubled Chickering ideal}, whose primary decomposition and variety we use to study the Chickering ideal. Throughout we fix a positive integer $n$ and associate to each family $A \to b$ with $b \in [n]$ and $A \subseteq [n] \setminus \{b\}$ two indeterminants $y_{A \to b}$, which we always take to be invertible, and $z_{A \to b}$, which is always non-invertible. Associated to each set $A \subseteq [n]$ we have an invertible  indeterminant $t_A$. We have the following toric ideal.

\begin{definition}
\label{def:ChickeringIdeal}
The \emph{Chickering ideal} $\chick_n$ is the kernel of the following two maps:

\begin{center}
    \begin{minipage}{0.45\linewidth}
    \begin{align*}
    \cc[z] &\to \cc[t_A^\pm: \emptyset \neq A \subseteq [n]] \\
    z^u &\mapsto t^{\varphi_n(u)}
    \end{align*}
    \end{minipage}
    \hfill
    \begin{minipage}{0.45\linewidth}
    \begin{align*}
    \cc[z] &\to \cc[t_A^\pm: A \subseteq [n]] \\
    z^u &\mapsto t^{\psi_n(u)}
    \end{align*}
    \end{minipage}
\end{center}
These maps have the same kernel by Proposition~\ref{prop: standard imset}. Associated to each directed graph $\cg$ we have a \emph{graph monomial} $z_\cg \coloneqq z^{\fvv_\cg}$.
\end{definition}

\begin{remark}
\label{rem:reasonForSingletons}
For each $\alpha \in \mathbb{Z}_{\geq 0}^{2^n-1}$, the fiber $z^{\varphi_n^{-1}(\alpha)}$ either consists only of graph monomials or contains no graph monomials. A fiber contains graph monomials if and only if $\alpha_b = 1$ for all $b \in [n]$, so that each node has exactly one family. Consequently when we include singleton sets in the definition of characteristic imsets, the fibers separate the family variable vectors representing directed graphs from those which do not represent graphs and partition the set of directed graphs into \emph{imset equivalence classes}.
\end{remark}

\begin{definition}
\label{def:CoveredEdgeFlipBinomial}
Let $b \neq c \in [n]$ and let $A \subseteq [n] \setminus \{b,c\}$. The \emph{covered edge flip binomial} associated to $A,b,c$ is $z_{A \to b} z_{A \cup b \to c} - z_{A \to c}z_{A \cup c \to b}$.
\end{definition}

Covered edge flips make a local change to a graph by reversing the orientation of one edge, which only changes the families of two vertices in the graph. The transformation of two families is reflected in the quadratic covered edge flip binomial; see Figure~\ref{fig:coveredEdgeFlip} to visualize this. Covered edge flip binomials enjoy the following special relationship with the Chickering ideal.

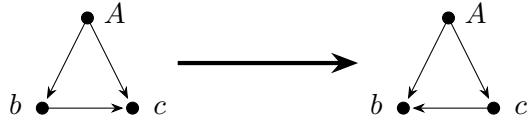
\begin{figure}
    \centering
\begin{tikzpicture}[scale = 1.2]
\begin{scope}[shift={(0,0)}]
\node[wB] at (0,0) {};
\node[wB] at (1,0) {};
\node[wB] at (0.5,1) {};

\node at (0.8,1.05) {$A$};
\node at (-0.3,0.05) {$b$};
\node at (1.3,0) {$c$};

\draw [-Stealth] (0.5,1)--(0.05,0.1);
\draw [-Stealth] (0.5,1)--(0.95,0.1);
\draw [-Stealth] (0,0)--(0.9,0);
\end{scope}

\draw[ultra thick, -Stealth] (1.5,0.5)--(3.5,0.5);

\begin{scope}[shift={(4,0)}]
\node[wB] at (0,0) {};
\node[wB] at (1,0) {};
\node[wB] at (0.5,1) {};

\node at (0.8,1.05) {$A$};
\node at (-0.3,0.05) {$b$};
\node at (1.3,0) {$c$};

\draw [-Stealth] (0.5,1)--(0.05,0.1);
\draw [-Stealth] (0.5,1)--(0.95,0.1);
\draw [-Stealth] (1,0)--(0.1,0);
\end{scope}

\end{tikzpicture}
    \caption{The binomial $z_{A \to b} z_{A \cup b \to c} - z_{A \to c}z_{A \cup c \to b}$ corresponds to this visual transformation of the graph.}
    \label{fig:coveredEdgeFlip}
\end{figure}

\begin{proposition}
\label{prop:LaurentIdeal}
For all $n\in \mathbb{N}$,
\begin{align*}\chick_n = \Big\langle z_{A \to b}z_{A \cup b \to c} - z_{A \to c} z_{A \cup c \to b} : b \neq c \in [n], A \subseteq [n] \setminus \{b,c\} \Big\rangle : \left(\prod\limits_{A \to b} z_{A \to b}\right)^\infty.\end{align*}
\end{proposition}

\begin{proof}
Let $e_{A \to b}$ denote the standard basis vector associated to $A \to b$ in $\mathbb{R}^{n \cdot 2^{n-1}}$ and let $\cl_n \subseteq \mathbb{Z}^{n \cdot 2^{n-1}}$ be the lattice spanned by all vectors of the form $e_{A \to b} + e_{A \cup b \to c} - e_{A \to c} - e_{A \cup c \to b}$. By \cite[Theorem~2.1]{eisenbudsturmfels}, to prove the result it suffices to show that $\cl_n = \ker_{\mathbb{Z}}(\varphi_n)$.

$(\subseteq)$ Let $S \subseteq [n]$ and let $\ind_{q}$ denote the indicator function for the event $q$. Then
\begin{align*}
&\varphi_n(e_{A \to b} + e_{A \cup b \to c} - e_{A \to c} - e_{A \cup c \to b})(S) \\
&= \ind_{b \in S} \cdot \ind_{S \subseteq A \cup b} + \ind_{c \in S} \cdot \ind_{S \subseteq A \cup b \cup c} - \ind_{c \in S} \cdot \ind_{S \subseteq A \cup c} - \ind_{b \in S} \cdot \ind_{S \subseteq A \cup b \cup c}.
\end{align*}
If $b,c \not\in S$, then clearly the expression evaluates to 0. If $b \in S$ and $c \not\in S$, then after simplification we obtain
\begin{align*}\ind_{S \subseteq A \cup b} - \ind_{S \subseteq A \cup b \cup c},\end{align*}
which is zero since $c \not\in S$. The case of $b \not\in S$ and $c \in S$ is similar. Lastly we consider the case where $b,c \in S$. In this case we obtain
\begin{align*} \ind_{S \subseteq A \cup b} + \ind_{S \subseteq A \cup b \cup c} - \ind_{S \subseteq A \cup c} - \ind_{S \subseteq A \cup b \cup c} = \ind_{S \subseteq A \cup b} - \ind_{S \subseteq A \cup c} = 0.\end{align*}
In all cases, $\varphi_n(e_{A \to b} + e_{A \cup b \to c} - e_{A \to c} - e_{A \cup c \to b})(S)=0$ and so $e_{A \to b} + e_{A \cup b \to c} - e_{A \to c} - e_{A \cup c \to b} \in \ker_{\mathbb{Z}}(\varphi_n)$. Hence $\cl_n \subseteq \ker_{\mathbb{Z}}(\varphi_n)$.

$(\supseteq)$ Let $v \in \ker_\mathbb{Z}(\varphi_n)$. First suppose that $v(A \cup c \to b) \neq 0$ for some $b,c \in [n]$ and $A \subseteq [n] \setminus \{b,c\}$ with $b < c$. Without loss of generality we may assume that $a < c$ for all $a \in A$. Let
\begin{align*}v' = v - v(A \cup c \to b)\left( e_{A \to c} + e_{A \cup c \to b} - e_{A \to b} - e_{A \cup b \to c}\right).\end{align*}
Then consider
\begin{align*}
S^\prime \coloneqq \sum\limits_{\substack{F \to g \\ \max(F) > g}} \#F \cdot \ind_{v^\prime(F \to g) \neq 0} \quad \mbox{and} \quad S \coloneqq \sum\limits_{\substack{F \to g \\ \max(F) > g}} \#F \cdot \ind_{v(F \to g) \neq 0}.
\end{align*}
We know that $v'$ and $v$ differ only in the $A \to b$, $A \to c$, $A \cup b \to c$, and $A \cup c \to b$ coordinates. Our sums $S^\prime$ and $S$ do not involve the $A \to c$ or $A \cup b \to c$ coordinates since these coordinates are topologically ordered. For $A \cup c \to b$, we know that $v(A \cup c \to b)\neq 0$ and $v^\prime(A \cup c \to b) = 0$, and so the sums $S^\prime$ and $S$ differ by $\#A+1$ in that summand. Lastly, we consider the summand associated to $A \to b$. This family corresponds to a summand only when $\max(A) > b$ and in that case the two sums differ by $\#A \cdot \left( \ind_{v(A \to b) \neq 0} - \ind_{v^\prime(A \to b) \neq 0} \right)$. So
\begin{align*}
S^\prime = S - (\#A +1) - (\#A) \cdot \ind_{\max(A) > b} \cdot (\ind_{v(A \to b) \neq 0} - \ind_{v^\prime(A \to b) \neq 0}) < S.\\
\end{align*}
Consequently we have found a nonnegative quantity that decreases when we modify $v$ to form $v'$. Proceeding inductively, we obtain a kernel element which only has topologically ordered coordinates nonzero after finitely many steps. Now assume $v \in \ker_\mathbb{Z}(\varphi_n)$ has all nonzero coordinates topologically ordered. By Proposition~\ref{prop:upperTriangular} we know that $v = 0$, and so $\ker_{\mathbb{Z}}(\varphi_n) \subseteq \cl_n$.
\end{proof}

\begin{example}
The following are the generators of $\chick_3$.

\begin{center}
\begin{tabular}{ c c c }
$z_{\emptyset \to 1} z_{1 \to 2} - z_{\emptyset \to 2} z_{2 \to 1}$, 
& $z_{\emptyset \to 1} z_{1 \to 3} - z_{\emptyset \to 3} z_{3 \to 1}$, 
& $z_{\emptyset \to 2} z_{2 \to 3} - z_{\emptyset \to 3} z_{3 \to 2}$, \\
$z_{3 \to 1} z_{13 \to 2} - z_{3 \to 2} z_{23 \to 1}$, 
& $z_{2 \to 1} z_{12 \to 3} - z_{2 \to 3} z_{23 \to 1}$, 
& $z_{1 \to 2} z_{12 \to 3} - z_{1 \to 3} z_{13 \to 2}$,  
\end{tabular}
\begin{tabular}{ c c }
$z_{1 \to 2}z_{2 \to 3}z_{3 \to 1} - z_{2 \to 1}z_{3 \to 2}z_{1 \to 3}$, & $z_{\emptyset \to 1}z_{3 \to 2}z_{1 \to 3} - z_{\emptyset \to 2}z_{2 \to 3}z_{3 \to 1}$,
\end{tabular}
\begin{tabular}{ c }
$z_{3 \to 1}z_{13 \to 2}z_{2 \to 3} - z_{2 \to 1}z_{3 \to 2}z_{12 \to 3}$.
\end{tabular}
\end{center}
The first two rows consist of the six covered edge flip binomials. The first binomial in the third row reverses the cycle $1 \to 2 \to 3 \to 1$ and this binomial cannot be expressed in terms of covered edge flips. Consequently we see that $\chick_3$ is not generated by covered edge flip binomials and moreover is not quadratically generated. The other two binomials also have graphical interpretations.
\end{example}

\begin{remark}
In general, $\chick_n$ has minimal generators in every degree $2,\dots,n$. Each covered edge flip binomial is a degree 2 minimal generator of $\chick_n$. For every $k \geq 3$ and $(i_1,\dots,i_k)$ a sequence of distinct elements of $[n]$, we obtain a minimal generator
\begin{align*}z_{i_1 \to i_2}z_{i_2 \to i_3} \cdots z_{i_{k-1} \to i_k} z_{i_k \to i_1} - z_{i_1 \to i_k}z_{i_k \to i_{k-1}} \cdots z_{i_3 \to i_2} z_{i_2 \to i_1}\end{align*}
since the associated fiber of $\varphi_n$ has size 2. These binomials arise from cycle reversals. More complicated minimal generators also exist; for instance in $\chick_5$ we have
\begin{align*}z_{24 \to 1}z_{35 \to 2}z_{15 \to 3}z_{25  \to 4}z_{14 \to 5} - z_{35 \to 1}z_{14 \to 2}z_{25 \to 3}z_{15 \to 4}z_{24 \to 5}.\end{align*}
The graphs associated to this minimal generator are depicted in Figure~\ref{fig:minimalGenerator}.

\begin{figure}
    \centering
\begin{tikzpicture}[scale=1.5]

\begin{scope}[shift={(0,0)}]
\node at (0,1.2) {\small 1 \normalsize};
\node at (-1.141,0.371) {\small 2 \normalsize};
\node at (-0.705,-0.971) {\small 3 \normalsize};
\node at (0.705,-0.971) {\small 4 \normalsize};
\node at (1.141,0.371) {\small 5 \normalsize};

\draw[-Stealth, shorten >=3pt,black] (-0.951,0.309) to (0,1);
\draw[-Stealth, shorten >=3pt,black] (0.588,-0.809) to (0,1);

\draw[-Stealth, shorten >=3pt,black] (-0.588,-0.809) to (-0.951,0.309);
\draw[-Stealth, shorten >=3pt,black] (0.951,0.309) to (-0.951,0.309);

\draw[-Stealth, shorten >=3pt,black] (0,1) to (-0.588,-0.809);
\draw[-Stealth, shorten >=3pt,red,thick] (0.951,0.309) to (-0.588,-0.809);

\draw[-Stealth, shorten >=3pt,black] (-0.951,0.309) to (0.588,-0.809);
\draw[-Stealth, bend left=10, shorten >=3pt,black] (0.951,0.309) to (0.588,-0.809);

\draw[-Stealth, shorten >=3pt,black] (0,1) to (0.951,0.309);
\draw[-Stealth, bend left=10, shorten >=3pt,black] (0.588,-0.809) to (0.951,0.309);

\node[wB] at (0,1) {};
\node[wB] at (-0.951,0.309) {};
\node[wB] at (-0.588,-0.809) {};
\node[wB] at (0.588,-0.809) {};
\node[wB] at (0.951,0.309) {};
\end{scope}

\draw[-Stealth,ultra thick] (1.5,0)--(2.5,0);

\begin{scope}[shift={(4,0)}]
\node at (0,1.2) {\small 1 \normalsize};
\node at (-1.141,0.371) {\small 2 \normalsize};
\node at (-0.705,-0.971) {\small 3 \normalsize};
\node at (0.705,-0.971) {\small 4 \normalsize};
\node at (1.141,0.371) {\small 5 \normalsize};

\draw[-Stealth, shorten >=3pt,black] (0,1) to (-0.951,0.309);
\draw[-Stealth, shorten >=3pt,black] (0,1) to (0.588,-0.809);

\draw[-Stealth, shorten >=3pt,black] (-0.951,0.309) to (-0.588,-0.809);
\draw[-Stealth, shorten >=3pt,black] (-0.951,0.309) to (0.951,0.309);

\draw[-Stealth, shorten >=3pt,black] (-0.588,-0.809) to (0,1);
\draw[-Stealth, shorten >=3pt,red,thick] (0.951,0.309) to (-0.588,-0.809);

\draw[-Stealth, shorten >=3pt,black] (0.588,-0.809) to (-0.951,0.309);
\draw[-Stealth, bend left=10, shorten >=3pt,black] (0.588,-0.809) to (0.951,0.309);

\draw[-Stealth, shorten >=3pt,black] (0.951,0.309) to (0,1);
\draw[-Stealth, bend left=10, shorten >=3pt,black] (0.951,0.309) to (0.588,-0.809);

\node[wB] at (0,1) {};
\node[wB] at (-0.951,0.309) {};
\node[wB] at (-0.588,-0.809) {};
\node[wB] at (0.588,-0.809) {};
\node[wB] at (0.951,0.309) {};
\end{scope}

\end{tikzpicture}
    \caption{An imset equivalence class of size 2. This class corresponds to the minimal generator $z_{24 \to 1}z_{35 \to 2}z_{15 \to 3}z_{25  \to 4}z_{14 \to 5} - z_{35 \to 1}z_{14 \to 2}z_{25 \to 3}z_{15 \to 4}z_{24 \to 5} \in \chick_5$. The graphs are related by reversing all edges except for the red one.}
    \label{fig:minimalGenerator}
\end{figure}
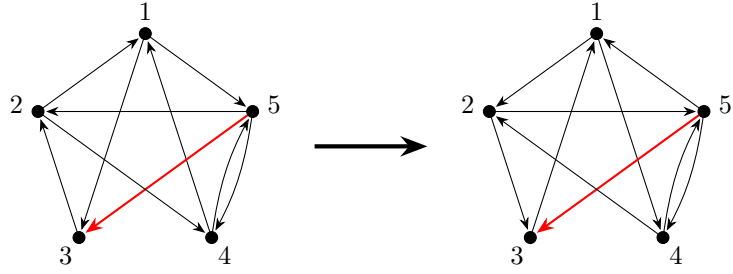
\end{remark}

\begin{definition}
\label{def:DoubledChickeringIdeal}
The \emph{doubled Chickering ideal} $\chick_n^\prime \subseteq \cc[y^\pm, z]$ is the binomial ideal 
\begin{align*}
\chick_n^\prime = \langle &y_{A \to b}y_{A \cup b \to c} - y_{A \to c}y_{A \cup c \to b}, \\
&y_{A \to b}z_{A \cup b \to c} - y_{A \to c}z_{A \cup c \to b}, \\
&z_{A \to b}z_{A \cup b \to c} - z_{A \to c}z_{A \cup c \to b}: b \neq c \in [n], A \subseteq [n] \setminus \{b,c\} \rangle.
\end{align*}
Furthermore, define
\begin{align*}
R_n = \cc[y^\pm,z] / \langle &y_{A \to b}y_{A \cup b \to c} - y_{A \to c}y_{A \cup c \to b}, \\ 
&y_{A \to b}z_{A \cup b \to c} - y_{A \to c}z_{A \cup c \to b} : b \neq c \in [n], A \subseteq [n] \setminus \{b,c\}\rangle.
\end{align*}
\end{definition}

We often think of $\chick_n^\prime$ as an ideal in $R_n$ instead of $\cc[y^\pm,z]$.

\begin{example}
If $n = 3$ then 
\begin{align*}
R_3 = \cc[y^\pm,z] /   \langle 
&y_{\emptyset \to 1} y_{1 \to 2} - y_{\emptyset \to 2} y_{2 \to 1},  y_{3 \to 1} y_{13 \to 2} - y_{3 \to 2} y_{23 \to 1}, \\
&y_{\emptyset \to 1} y_{1 \to 3} - y_{\emptyset \to 3} y_{3 \to 1}, y_{2 \to 1} y_{12 \to 3} - y_{2 \to 3} y_{23 \to 1}, \\
&y_{\emptyset \to 2} y_{2 \to 3} - y_{\emptyset \to 3} y_{3 \to 2}, y_{1 \to 2} y_{12 \to 3} - y_{1 \to 3} y_{13 \to 2}, \\
&y_{\emptyset \to 1} z_{1 \to 2} - y_{\emptyset \to 2} z_{2 \to 1},  y_{3 \to 1} z_{13 \to 2} - y_{3 \to 2} z_{23 \to 1}, \\
&y_{\emptyset \to 1} z_{1 \to 3} - y_{\emptyset \to 3} z_{3 \to 1}, y_{2 \to 1} z_{12 \to 3} - y_{2 \to 3} z_{23 \to 1}, \\
&y_{\emptyset \to 2} z_{2 \to 3} - y_{\emptyset \to 3} z_{3 \to 2}, y_{1 \to 2} z_{12 \to 3} - y_{1 \to 3} z_{13 \to 2}
\rangle.
\end{align*}
As an ideal in $R_3$, 
\begin{align*}
\chick_3^\prime =\langle 
z_{\emptyset \to 1} z_{1 \to 2} - z_{\emptyset \to 2} z_{2 \to 1},  z_{3 \to 1} z_{13 \to 2} - z_{3 \to 2} z_{23 \to 1}, \\
z_{\emptyset \to 1} z_{1 \to 3} - z_{\emptyset \to 3} z_{3 \to 1}, z_{2 \to 1} z_{12 \to 3} - z_{2 \to 3} z_{23 \to 1}, \\
z_{\emptyset \to 2} z_{2 \to 3} - z_{\emptyset \to 3} z_{3 \to 2}, z_{1 \to 2} z_{12 \to 3} - z_{1 \to 3} z_{13 \to 2}
\rangle.
\end{align*}
\end{example}

\begin{proposition}
$R_n$ is an integral domain.
\end{proposition}

\begin{proof}
By Proposition~\ref{prop:LaurentIdeal} we know $I \coloneqq \langle y_{A \to b}y_{A \cup b \to c} - y_{A \to c}y_{A \cup c \to b} : b,c \in [n], A \subseteq [n] \setminus \{b,c\} \rangle$ is toric, and so $\cc[y^\pm,z] / I$ is an integral domain. Up to isomorphism, this quotient ring is $S[z]$ for some integral domain $S$. It remains to quotient by binomials of the form $y_{A \to b}z_{A \cup b \to c} - y_{A \to c}z_{A \cup c \to b}$, which are homogeneous linear forms in the $z$-variables with coefficients that are units in $S[z]$. Quotienting an integral domain by these linear forms yields an integral domain.
\end{proof}

We now turn towards the problem of radicality of the doubled Chickering ideal. Neither $R_n$ nor $\cc[y^\pm,z]$ are polynomial rings, so Gr\"obner basis techniques do not apply here. We instead compute a primary decomposition by finding a sequence of splitting polynomials.

\begin{proposition}
\label{prop:newGenerators}
As an ideal in $R_n$,
\begin{align*}\chick_n^\prime = \langle z_{A \cup c \to b} \left( z_{A \to b} y_{A \to c} - z_{A \to c} y_{A \to b} \right): b \neq c \in [n], A \subseteq [n] \setminus \{b,c\} \rangle.\end{align*}
\end{proposition}

\begin{proof}
Clearly $\chick_n^\prime = \langle  z_{A \to b} z_{A \cup b \to c} - z_{A \to c} z_{A \cup c \to b} : b \neq c \in [n], A \subseteq [n] \setminus \{b,c\} \rangle$ since all other generators are $0$ in $R_n$. Any generator $z_{A \to b}z_{A \cup b \to c} - z_{A \to c}z_{A \cup c \to b}$ can be multiplied by the unit $y_{A \to b}$ without changing the ideal. Using the fact that $y_{A \to b} z_{A \cup b \to c} = y_{A \to c} z_{A \cup c \to b}$ we obtain
\begin{align*}y_{A \to b}\left(z_{A \to b} z_{A \cup b \to c} - z_{A \to c} z_{A \cup c \to b}\right) = z_{A \cup c \to b} \left( z_{A \to b} y_{A \to c} - z_{A \to c} y_{A \to b} \right).\end{align*}
\end{proof}

\begin{definition}
Let $\mathfrak{B}_n$ be the subposet of the Boolean lattice $2^{[n]}$ on the set $\{A \in 2^{[n]}: \#A \geq 2\}$. Let $F$ be an upward closed subset of $\mathfrak{B}_n$ and let $E \subseteq F$. We define an ideal $I_{E,F} \subseteq R_n$ to be
\begin{equation}
\label{equation:ief}
\begin{split}
I_{E, F} &= \langle z_{A \cup c \to b} \left( z_{A \to b} y_{A \to c} - z_{A \to c}y_{A \to b} \right) : A \cup \{b,c\} \in \mathfrak{B}_n \setminus F \rangle \\
&+ \langle z_{A \to b}: A \cup \{b\} \in E \rangle \\
&+ \langle  z_{A \to b} y_{A \to c} - z_{A \to c}y_{A \to b} : A \cup \{b,c\} \in F \setminus E \rangle.
\end{split}
\end{equation}
In the case of $F = \mathfrak{B}_n$, we let $P_E = I_{E, \mathfrak{B}_n}$.
\end{definition}

\begin{proposition}
\label{lemma:primaryDecomposition} 
For all pairs $E \subseteq F$, $I_{E,F}$ is a radical ideal. Furthermore,
\begin{align*}\chick_n^\prime = \bigcap\limits_{E \subseteq \mathfrak{B}_n} P_E 
\end{align*}
yields a primary decomposition of $\chick_n^\prime$ into toric ideals.
\end{proposition}

\begin{proof}
In the case of $F = \mathfrak{B}_n$, $P_E = I_{E,F}$ is generated by homogeneous linear forms in the $z$-variables, and so $P_E$ is prime whether it is considered as an ideal in $R_n$ or $\cc[y^\pm,z]$. By Proposition~\ref{prop:newGenerators} we know that in the case that $E = F = \emptyset$, $I_{\emptyset, \emptyset} = \chick_n^\prime$. Hence it remains to show that for any upward closed $F$, $E \subseteq F$, and maximal element $D \in \mathfrak{B}_n \setminus F$,
\begin{align*}
I_{E,F} = I_{E \cup \{D\}, F \cup \{D\}} \cap I_{E, F \cup \{D\}}.
\end{align*}
Let $b^* \in D$ and $A^* = D \setminus b^*$. It suffices to show that 

\begin{enumerate}[label=\textbf{Claim \arabic*.}, leftmargin=*, align=left]
    \item  $I_{E \cup \{D\}, F \cup \{D\}} = I_{E, F} + \langle z_{A^* \to b^*} \rangle$ and
    \item $I_{E, F \cup \{D\}} = I_{E, F} : z_{A^* \to b^*} = I_{E, F} : z_{A^* \to b^*}^\infty$.
\end{enumerate}

We first prove Claim 1. First let $c^* \in A^*$ and let $\bar{A} = A^* \setminus c^\star$. Then $z_{\bar{A} \cup c^* \to b^*}, y_{\bar{A} \to b^*} z_{\bar{A} \cup b^* \to c^*} - y_{\bar{A} \to c^*} z_{\bar{A} \cup c^* \to b^*} \in I_{E,F} + \langle z_{A^* \to b^*} \rangle$. Since $y_{\bar{A} \to b^*}$ and $y_{\bar{A} \to c^*}$ are units, it follows that $z_{\bar{A} \cup b^* \to c^*} \in I_{E,F} + \langle z_{A^* \to b^*} \rangle$. Consequently
\begin{equation}
\label{equation:iefSum}
\begin{split}
I_{E,F} + \langle z_{A^* \to b^*} \rangle 
&= I_{E,F} + \langle z_{A \to b}: b \in D, A= D \setminus b \rangle \\
&= \langle z_{A \cup c \to b} \left( z_{A \to b} y_{A \to c} - z_{A \to c}y_{A \to b} \right) : A \cup \{b,c\} \in \mathfrak{B}_n \setminus F \rangle \\
&+ \langle z_{A \to b}: A \cup b \in E \cup \{D\} \rangle \\
&+ \langle  z_{A \to b} y_{A \to c} - z_{A \to c}y_{A \to b} : A \cup \{b,c\} \in F \setminus E \rangle.
\end{split}
\end{equation}
The second and third ideal summands in the definition of $I_{E \cup \{D\},F \cup \{D\}}$ agree with~\Cref{equation:iefSum} since $F \setminus E = (F \cup \{D\}) \setminus (E \cup \{D\})$. For the first summand, if $b \neq c \in D$ and $A = D \setminus \{b,c\}$, then $z_{A \cup c \to b} \left( z_{A \to b} y_{A \to c} - z_{A \to c}y_{A \to b} \right)$ is redundant since $z_{A \cup c \to b}$ lies in the second ideal summand in~\Cref{equation:iefSum}. Thus $I_{E,F} + \langle z_{A^* \to b^*} \rangle = I_{E \cup \{D\}, F \cup \{D\}}$.

We now prove Claim 2. Let $k$ be any positive integer. For any $b,c \in D$, $y_{D \setminus \{b,c\} \to c}z_{D \setminus b \to b} - y_{D \setminus \{b,c\} \to b}z_{D \setminus c \to c} \in I_{E,F}$. Consequently, if $z_{D \setminus b \to b} \cdot f \in I_{E,F}$ for some polynomial $f$, then
\begin{align*}z_{D \setminus c \to c}^k \cdot f = \frac{y_{D \setminus \{b,c\} \to c}^k}{y_{D \setminus \{b,c\} \to b}^k} z_{D \setminus b \to b}^k \cdot f \in I_{E,F}.\end{align*}
Hence
\begin{align*}I_{E,F} : z_{A^* \to b^*}^k = I_{E,F} : \left( \prod\limits_{b \in D} z_{D \setminus \{b\} \to b}\right)^k.\end{align*}
For any $A \cup \{b,c\} = D$, we know that $I_{E,F}$ contains $z_{A \cup c \to b} \left( z_{A \to b} y_{A \to c} - z_{A \to c}y_{A \to b} \right)$ and so $z_{A \to b} y_{A \to c} - z_{A \to c}y_{A \to b} \in I_{E, F} : z_{A^* \to b^*}^k$. Thus
\begin{equation}
\label{equation:iefQuotient}
\begin{split}
I_{E, F}:z_{A^* \to b^*}^k &= \Bigl( \langle z_{A \cup c \to b} \left( z_{A \to b} y_{A \to c} - z_{A \to c}y_{A \to b} \right) : A \cup \{b,c\} \in \mathfrak{B}_n \setminus (F \cup \{D\}) \rangle \\
&+ \langle z_{A \to b}: A \cup \{b\} \in E \rangle \\
&+ \langle  z_{A \to b} y_{A \to c} - z_{A \to c}y_{A \to b} : A \cup \{b,c\} \in (F \cup \{D\}) \setminus E \rangle\Bigl): z_{A^* \to b^*}^k \\
&= I_{E, F \cup \{D\}} : z_{A^* \to b^*}^k.
\end{split}
\end{equation}
In $I_{E,F \cup \{D\}}$, the variable $z_{A^* \to b^*}$  only appears in homogeneous linear forms in the $z$-variables. Taking the quotient of $R_n$ by these generators collects $z$-variables that appear in the same linear form into equivalences classes $w_i$, and we obtain a ring which is isomorphic to
\begin{align*}
\left( \cc[y^\pm] / \langle y_{A \to b}y_{A \cup \{b\} \to c} - y_{A \to c}y_{A \cup \{c\} \to b} : b \neq c \in [n], A \subseteq [n] \setminus \{b,c\} \rangle \right)[w_1,\dots,w_\ell].
\end{align*}
In this ring, no generator of $I_{E,F \cup \{D\}}$ contains the equivalence class of $z_{A^* \to b^*}$ and so
\begin{align*}I_{E,F}:z_{A^* \to b^*}^k = I_{E, F \cup \{D\}} : z_{A^* \to b^*}^k = I_{E,F \cup \{D\}}.\end{align*}
Since this equation holds for all $k$, Claim 2 holds.
\end{proof}

\begin{proposition}
\label{prop:closedUnderTorusAction}
Consider the torus action $(\cc^*)^{2^n+n} \, \acts \, (\left(\cc^*\right)^{n \cdot 2^{n-1}} \times \cc^{n \cdot 2^{n-1}})$ given by
\begin{align*}
\left( (\vec{t},\vec{r}) \, \acts \, (y,z)\right)_{A \to b} &= \left( \frac{t_{A \cup b}}{t_A} \cdot y_{A \to b}, \, r_{\#A+1} \cdot \frac{t_{A \cup b}}{t_A} \cdot  z_{A \to b} \right),
\end{align*}
where the coordinates of $\vec{t}$ are indexed by all sets $A\subseteq [n]$ and $\vec{r} = (r_1,\dots, r_n)$. The variety $V(\chick_n^\prime)$ is closed under this action.
\end{proposition}

\begin{proof}
Let $(y,z) \in V(\chick_n^\prime)$. Acting on $(y,z)$ by some $(1,r)$ with $r \in \left(\mathbb{C}^*\right)^n$, scales each $z_{A \to b}$ by $r_{\#A+1}$. Since the generators of $\chick_n^\prime$ have the same number of $z_{A \to b}$ of any given index size in each monomial, we know that $(1,r) \, \acts \, (y,z) \in V(\chick_n^\prime)$.

It remains to consider actions $(y^*,z^*) := (t,1) \, \acts \, (y,z)$ where $t \in \left(\mathbb{C}^* \right)^{2^n}$. For all $b\neq c \in [n]$ and $A\subseteq [n]\setminus \{b,c\}$,
\begin{align*}
y_{A \to b}^* z_{A \cup b \to c}^* - y_{A \to c}^* z_{A \cup c \to b}^* = \frac{t_{A \cup b \cup c}}{t_{A}} \left( y_{A \to b} z_{A \cup b \to c} - y_{A \to c} z_{A \cup c \to b} \right) = 0.
\end{align*}
Analogously $y_{A \to b}^* y_{A \cup b \to c}^* - y_{A \to c}^* y_{A \cup c \to b}^* = 0$ and $z_{A \to b}^* z_{A \cup b \to c}^* - z_{A \to c}^* z_{A \cup c \to b}^* = 0$.
\end{proof}

\begin{lemma}
\label{lemma:doubledChickeringIdealSaturation}
The variety $V(P_\emptyset)$ is the Zariski closure of the torus orbit generated by $\vec{1} \in \cc^{2n \cdot 2^{n-1}}$ under the torus action $(\cc^*)^{2^n+n} \, \acts \, (\left(\cc^*\right)^{n \cdot 2^{n-1}} \times \cc^{n \cdot 2^{n-1}})$ given by
\begin{align*}
\left( (\vec{t},\vec{r}) \, \acts \, (y,z)\right)_{A \to b} &= \left( \frac{t_{A \cup b}}{t_A} \cdot y_{A \to b}, \, r_{\#A+1} \cdot \frac{t_{A \cup b}}{t_A} \cdot  z_{A \to b} \right),
\end{align*}
where the coordinates of $\vec{t}$ are indexed by all sets $A\subseteq [n]$ and $\vec{r} = (r_1,\dots, r_n)$.
\end{lemma}

\begin{proof}
Let $J$ denote the toric ideal associated to the torus action defined in Lemma~\ref{lemma:doubledChickeringIdealSaturation}. It is straightforward to show that $P_\emptyset \subseteq J$, and so it remains to show that $J \subseteq P_\emptyset$. By \Cref{lemma:primaryDecomposition} we know that both $J$ and $P_\emptyset$ are prime ideals, and so equivalently we must show  $V(P_\emptyset) \subseteq V(J)$.

Fix $(y,z) \in V(P_\emptyset)$ and let $(y^*,z^*) := (t,1) \, \acts \, (y,z)$ for some $t \in \left(\cc^*\right)^{2^n}$. First observe that $z_{A \to b}^* y_{A  \to c}^* - z_{A \to c}^* y_{A  \to b}^* = 0$ by the same argument in \Cref{prop:closedUnderTorusAction}, and so $V(P_\emptyset)$ is closed under the torus action.

We know that $P_\emptyset$ and $J$ are both prime. To see that $V(P_\emptyset) \subseteq V(J)$, we only need to show that if $(y,z) \in \left(\cc^*\right)^{2n \cdot 2^{n-1}} \cap V(P_\emptyset)$, then $(y,z)$ is in the torus orbit of $(1,1)$. By \Cref{prop:LaurentIdeal} there exists $t \in \left(\cc^*\right)^{2^n}$ such that $(t,1) \, \acts \, (y,z) = (1,z^*)$ for some $z^*$. Since $(1,z^*) \in V(P_\emptyset)$, we know that $z_{A \to b}^* = z_{C \to d}^*$ whenever $\#A = \#C$. Acting on $(1, z^*)$ by the appropriate vector $(1,r)$ scales $z_{A \to b}^*$ to 1 for all $A,b$.
\end{proof}

\begin{example}
In the case of $n=3$, the variety $V(P_\emptyset)$ is the toric variety associated to the matrix
\begin{align*}
\bordermatrix{ & z_{\emptyset \to 1} & z_{2 \to 1} & z_{3 \to 1} & z_{23 \to 1} & z_{\emptyset \to 2} & \cdots & y_{\emptyset \to 3} & y_{1 \to 3} & y_{2 \to 3} & y_{12 \to 3} \cr
t_{\emptyset}   & -1 & 0 & 0 & 0 & -1 & \cdots & -1 & 0 & 0 & 0 \cr
t_{1}   & 1 & 0 & 0 & 0 & 0 & \cdots & 0 & -1 & 0 & 0 \cr
t_{2}   & 0 & -1 & 0 & 0 & 1 &        & 0 & 0 & -1 & 0 \cr
t_{3}   & 0 & 0 & -1 & 0 & 0 &        & 1 & 0 & 0 & 0 \cr
t_{12}  & 0 & 1 & 0 & 0 & 0 &        & 0 & 0 & 0 & -1 \cr
t_{13}  & 0 & 0 & 1 & 0 & 0 &        & 0 & 1 & 0 & 0 \cr
t_{23}  & 0 & 0 & 0 & -1 & 0 &        & 0 & 0 & 1 & 0 \cr
t_{123} & 0 & 0 & 0 & 1 & 0 &        & 0 & 0 & 0 & 1 \cr
r_1     & 1 & 0 & 0 & 0 & 1 &        & 0 & 0 & 0 & 0 \cr 
r_2     & 0 & 1 & 1 & 0 & 0 &        & 0 & 0 & 0 & 0 \cr 
r_3     & 0 & 0 & 0 & 1 & 0 & \cdots & 0 & 0 & 0 & 0 \cr 
}. \qquad
 \end{align*}
The variety is cut out by the binomials

\begin{center}
\begin{tabular}{l l l l}
$z_{\emptyset \to 1}z_{1 \to 2} - z_{\emptyset \to 2}z_{2 \to 1}$, & $z_{\emptyset \to 1}z_{1 \to 3} - z_{\emptyset \to 3}z_{3 \to 1}$, & $z_{\emptyset \to 2}z_{2 \to 3} - z_{\emptyset \to 3}z_{3 \to 2}$,  \\
$y_{\emptyset \to 1}z_{1 \to 2} - y_{\emptyset \to 2}z_{2 \to 1}$, & $y_{\emptyset \to 1}z_{1 \to 3} - y_{\emptyset \to 3}z_{3 \to 1}$, & $y_{\emptyset \to 2}z_{2 \to 3} - y_{\emptyset \to 3}z_{3 \to 2}$,  \\
$z_{\emptyset \to 1}y_{1 \to 2} - z_{\emptyset \to 2}y_{2 \to 1}$, & $z_{\emptyset \to 1}y_{1 \to 3} - z_{\emptyset \to 3}y_{3 \to 1}$, & $z_{\emptyset \to 2}y_{2 \to 3} - z_{\emptyset \to 3}y_{3 \to 2}$,  \\
$y_{\emptyset \to 1}y_{1 \to 2} - y_{\emptyset \to 2}y_{2 \to 1}$, & $y_{\emptyset \to 1}y_{1 \to 3} - y_{\emptyset \to 3}y_{3 \to 1}$, & $y_{\emptyset \to 2}y_{2 \to 3} - y_{\emptyset \to 3}y_{3 \to 2}$,  \\
$z_{3 \to 1}z_{13 \to 2} - z_{3 \to 2}z_{23 \to 1}$, & $z_{2 \to 1}z_{12 \to 3} - z_{2 \to 3}z_{23 \to 1}$, & $z_{1 \to 2}z_{12 \to 3} - z_{1 \to 3}z_{13 \to 2}$,  \\
$y_{3 \to 1}z_{13 \to 2} - y_{3 \to 2}z_{23 \to 1}$, & $y_{2 \to 1}z_{12 \to 3} - y_{2 \to 3}z_{23 \to 1}$, & $y_{1 \to 2}z_{12 \to 3} - y_{1 \to 3}z_{13 \to 2}$,  \\
$z_{3 \to 1}y_{13 \to 2} - z_{3 \to 2}y_{23 \to 1}$, & $z_{2 \to 1}y_{12 \to 3} - z_{2 \to 3}y_{23 \to 1}$, & $z_{1 \to 2}y_{12 \to 3} - z_{1 \to 3}y_{13 \to 2}$,  \\
$y_{3 \to 1}y_{13 \to 2} - y_{3 \to 2}y_{23 \to 1}$, & $y_{2 \to 1}y_{12 \to 3} - y_{2 \to 3}y_{23 \to 1}$, & $y_{1 \to 2}y_{12 \to 3} - y_{1 \to 3}y_{13 \to 2}$. 
\end{tabular}
\end{center}

\end{example}

\begin{theorem}
\label{thm:eliminationIdeal}
For all $n\in \mathbb{N}$, $\chick_n = \chick_n^\prime \cap \cc[z]$.
\end{theorem}

\begin{proof}
Throughout we use the action obtained from the map $\psi_n:\mathbb{R}^{n \cdot 2^{n-1}} \to \mathbb{R}^{2^n}$ from Definition~\ref{def:standardimsets} defined by
\begin{align*}\left(t \, \acts \, (y,z)\right)_{A \to b} = \left( \frac{t_{A \cup b}}{t_A} y_{A \to b}, \frac{t_{A \cup b}}{t_A} z_{A \to b} \right).\end{align*}

By \Cref{lemma:primaryDecomposition} and \Cref{lemma:doubledChickeringIdealSaturation} we know $\chick_n^\prime \cap \cc[z] \subseteq P_\emptyset \cap \cc[z] = \chick_n$. We know that $\chick_n$ is prime and by \Cref{lemma:primaryDecomposition} we know that $\chick_n^\prime$ is radical. Consequently, to prove the reverse inclusion it suffices to show that $V(\chick_n^\prime \cap \cc[z]) \subseteq V(\chick_n)$.

Let $z^* \in V(\chick_n^\prime \cap \cc[z])$. Then there exists $y \in \left(\cc^*\right)^{n \cdot 2^{n-1}}$ such that $(y,z^*) \in V(\chick_n^\prime)$. Since $y \in V(\chick_n) \cap \left(\cc^*\right)^{n \cdot 2^{n-1}}$, we know that there exist $t \in \left(\cc^*\right)^{2^n}$  and $z \in \cc^{n \cdot 2^{n-1}}$ such that $t \, \acts \,(y,z^*) = (1,z)$. By \Cref{prop:closedUnderTorusAction} we have
\begin{align*}z \in W \coloneqq V(z_{A \cup b \to c} - z_{A \cup c \to b}, z_{A \to b}z_{A \cup b\to c} - z_{A \to c}z_{A \cup c \to b}: b \neq c \in [n], A \subseteq [n] \setminus \{b,c\}).\end{align*}
To show that $V(\chick_n^\prime \cap \cc[z]) \subseteq V(\chick_n)$ it suffices to show that $W \subseteq V(\chick_n)$ since $V(\chick_n^\prime \cap \cc[z])$ is the closure of $W$ under the torus action. To do this we will show the following two claims:

\begin{enumerate}[label=\textbf{Claim \arabic*.}, leftmargin=*, align=left]
    \item  Every $z \in W$ is in the torus orbit of some $0/1$-vector in $W$ and
    \item $\{0,1\}^{n \cdot 2^{n-1}} \cap W \subseteq V(\chick_n)$.
\end{enumerate}
For Claim 1, let $z \in W$. It suffices to construct a sequence
\begin{align*}z=:z^{(0)}, z^{(1)},\dots,z^{(n)}\end{align*}
such that for some $t^{(1)},\dots,t^{(n)} \in \left(\cc^*\right)^{2^n}$ we have that for all $k \in [n]$
\begin{enumerate}[label=\textbf{Subclaim \arabic*.}, leftmargin=*, align=left]
    \item $t^{(k)} \, \acts \, z^{(k-1)} = z^{(k)}$, 
    \item $z_{A \to b}^{(k)} \in \{0,1\}$ if $\#A < k$, and
    \item $z^{(k)} \in W$.
\end{enumerate}
Let $z^{(0)}\coloneqq z$. We do not need to consider Subclaim 1 for the base case since there is no preceding $z^{(k)}$. Subclaim 2 vacuously holds since $\#A$ cannot be negative for any family. Lastly, for Subclaim 3 we know that $z^{(0)} \in W$ since $z \in W$ by assumption.

We inductively define
\begin{align*}
t^{(k)}_S \coloneqq \begin{cases} \left(z_{S \setminus b \to b}^{(k-1)} \right)^{-1} &\mbox{if}~ \#S = k ~\mbox{and}~ z_{S \setminus b \to b}^{(k-1)} \neq 0 \\ 1 &\mbox{otherwise}\end{cases} \quad \mbox{and} \quad z^{(k)}\coloneqq t^{(k)} \, \acts \, z^{(k-1)}.
\end{align*}
By induction we know $z^{(k-1)} \in W$ and so $z_{S \setminus b \to b}^{(k-1)}$ is independent of the choice of $b \in S$. It follows that $t^{(k)}$ is well-defined. Thus Subclaim 1 holds for $z^{(k)}$ and $z^{(k)}$ is well-defined.

For Subclaim 2, first note that $z^{(k)}_{A \to b} = z^{(k-1)}_{A \to b}$ whenever $\#A \not\in \{k-1,k\}$ since $t^{(k)}$ only changes $z^{(k-1)}_{A \to b}$ when $\#A \in \{k-1,k\}$. Consequently we only need to check that $z_{A \to b}^{(k)} \in \{0,1\}$ when $\#A = k-1$. In this case
\begin{align*}
z_{A \to b}^{(k)} = \frac{t_{A \cup b}^{(k)}}{t_A^{(k)}} z_{A \to b}^{(k-1)} = t_{A \cup b}^{(k)} z_{A \to b}^{(k-1)} = \begin{cases} 1 &\mbox{if}~z_{A \to b}^{(k-1)} \neq 0, \\ 0 &\mbox{otherwise}. \end{cases}\end{align*}

For Subclaim 3, we know that $V(z_{A \to b}z_{A \cup b \to c} - z_{A \to c}z_{A \cup c \to b}: b \neq c \in [n], A \subseteq [n] \setminus \{b,c\})$ is closed under the torus action, so we only need to check the constraints of the form $z_{A \cup b \to c}^{(k)} = z_{A \cup c \to b}^{(k)}$. When $\#A \not\in \{k-2,k-1\}$, $z_{A \cup b \to c}^{(k)} = z_{A \cup b \to c}^{(k-1)} =z_{A \cup c \to b}^{(k-1)} = z_{A \cup c \to b}^{(k)}$.

When $\#A = k-2$, we know $z_{A \cup b \to c}^{(k)} \in \{0,1\}$ by Subclaim 2 and the value is determined by the sparsity of $z^{(k-1)}$. Hence $z_{A \cup b \to c}^{(k)} = z_{A \cup c \to b}^{(k)}$.

Lastly, we consider the case of $\#A = k-1$. If $z_{A \cup b \to c}^{(k-1)} = z_{A \cup c \to b}^{(k-1)} = 0$, then $z_{A \cup b \to c}^{(k)} = z_{A \cup c \to b}^{(k)} = 0$. Otherwise $z_{A \cup b \to c}^{(k-1)} = z_{A \cup c \to b}^{(k-1)} \neq 0$, and so $z_{A \to b}^{(k-1)} = z_{A \to c}^{(k-1)}$. If $z_{A \to b}^{(k-1)}= 0$, then 
\begin{align*}
z_{A \cup b \to c}^{(k)} = \frac{t_{A \cup b \cup c}^{(k)}}{t_{A \cup b}^{(k)}} z_{A \cup b \to c}^{(k-1)} = z_{A \cup b \to c}^{(k-1)} = z_{A \cup c \to b}^{(k-1)} = \frac{t_{A \cup b \cup c}^{(k)}}{t_{A \cup c}^{(k)}} z_{A \cup c \to b}^{(k-1)} = z_{A \cup c \to b}^{(k)}.
\end{align*}
Otherwise $z_{A \to b}^{(k-1)} \neq 0$, and 
\begin{align*}z_{A \cup b \to c}^{(k)} = \frac{t_{A \cup b \cup c}^{(k)}}{t_{A \cup b}^{(k)}} z_{A \cup b \to c}^{(k-1)} = z_{A \to b}^{(k-1)} z_{A \cup b \to c}^{(k-1)} = z_{A \to c}^{(k-1)} z_{A \cup c \to b}^{(k-1)} = \frac{t_{A \cup b \cup c}^{(k)}}{t_{A \cup c}^{(k)}} z_{A \cup c \to b}^{(k-1)} = z_{A \cup c \to b}^{(k)}.\end{align*}
Hence $z^{(k)} \in W$, and in particular $z^{(n)} \in W \cap \{0,1\}^{n \cdot 2^{n-1}}$.

We now show Claim 2. Throughout the remainder of the proof, we assume $z \in W \cap \{0,1\}^{n \cdot 2^{n-1}}$. We will find a path to $z$ through the parameterized part of $V(\chick_n)$. Let $\tau$ be a parameter taking on values in $(0,1]$. Define
\begin{align*}
t_\emptyset \coloneqq 1 \quad \mbox{and} \quad t_{A \cup b} \coloneqq \begin{cases} t_A &\mbox{if}~ z_{A \to b} = 1, \\ \tau^{\#A+1}&\mbox{if}~z_{A \to b} = 0 .\end{cases}
\end{align*}
We must show that $t$ is well-defined and that $\lim\limits_{\tau \to 0^+} \frac{t_{A \cup b}}{t_A} = z_{A \to b}$.

We first show that $t$ is well-defined. For the base case, let $b \in [n]$. Clearly $t_b$ is well-defined since there is only one choice of distinguished element in $\{b\}$. For the induction step, suppose that for all $A \to b$ with $\#A < k-1$, we know $t_{A \cup b}$ is well-defined. Now fix $ b \neq c \in [n]$ and $A \subseteq [n] \setminus \{b,c\}$ with $\#A = k-2$.

If $z_{A \cup b \to c}=0$, then $z_{A \cup c \to b} = 0$ since $z \in W$. Thus the definition of $t_{A \cup b \cup c} \coloneqq \tau^{\#A+2}$ is independent of any distinguished element in the indexing set. Otherwise $z_{A \cup b \to c} = 1$. Since $z_{A \cup b \to c} = z_{A \cup c \to b}$ we obtain two definitions of $t_{A \cup b \cup c}$:
\begin{align*}t_{A \cup b \cup c}\coloneqq t_{A \cup b} \quad \mbox{and} \quad t_{A \cup b \cup c} \coloneqq t_{A \cup c}.\end{align*}
Since $z_{A \cup b \to c}(z_{A \to b} - z_{A \to c}) = 0$ and $z_{A \cup b \to c} \neq 0$, we know $z_{A \to b} = z_{A \to c}$ and so $t_{A \cup b} = t_{A \cup c}$. It follows that $t$ is well-defined.

We now compute the limit. For some $\alpha > 0$, 
\begin{align*}\lim\limits_{\tau \to 0^+} \frac{t_{A \cup b \cup c}}{t_{A \cup b}} = \lim\limits_{\tau \to 0^+}\left( \begin{cases} 1 &\mbox{if}~ z_{A \cup b \to c} = 1 \\ \tau^\alpha &\mbox{if}~ z_{A \cup b \to c} = 0 \end{cases} \right) = z_{A \cup b \to c}.\end{align*}
Since there is a path in $V(\chick_n)$ with $z$ as a limit point, we have $z \in V(\chick_n)$. It follows that $W \subseteq V(\chick_n)$ and $V(\chick_n^\prime \cap \cc[z]) \subseteq V(\chick_n)$.
\end{proof}

\section{Covariance Equivalence}
\label{sec:matrixMoves}

In this section we prove the main theorem of this article, which is the following result.

\begin{theorem} \label{thm: main}
    Let $\mathcal{G}$ and $\mathcal{H}$ be two directed graphs on the vertex set $[n]$ such that $\cimv_\mathcal{G} = \cimv_\mathcal{H}$. Then $\mathcal{G}$ and $\mathcal{H}$ are covariance equivalent. 
\end{theorem}

\begin{proof}
Let $\cg, \ch$ be directed graphs with $c_\cg = c_\ch$. By Definition~\ref{def:ChickeringIdeal}, $z_\cg - z_\ch \in \chick_n$ and then by Theorem~\ref{thm:eliminationIdeal}, $z_\cg - z_\ch \in \chick_n^\prime \cap \cc[z]$. It follows that $z_\cg$ can be transformed into $z_\ch$ via the binomials
\begin{equation}
\label{eq:binomialMoves}
1-\frac{y_{A \to c}y_{A \cup c \to b}}{y_{A \to b}y_{A \cup b \to c}}, \quad\quad z_{A \cup b \to c} -\frac{y_{A \to c}}{y_{A \to b}}z_{A \cup c \to b}, \quad\quad z_{A \to b}z_{A \cup b \to c} - z_{A \to c}z_{A \cup c \to b}
\end{equation}
for $b\neq c \in [n]$ and sets $A \subseteq [n] \setminus \{b,c\}$. In other words, 
\begin{align} \label{eq: binomial in doubled chickering}   
    z_\mathcal{G} - z_\mathcal{H} = \sum_{i=1}^L  z^{p_i} y^{q_i} (z^{u_i^+} - z^{u_i^-}y^{v_i}),
  \end{align}
where for all $i\in L$, $p_i,u_i^+,u_i^- \in (\mathbb{Z}_{\geq 0})^{n\cdot {2^{n-1}}}$, $q_i,v_i \in \mathbb{Z}^{n\cdot {2^{n-1}}}$, $z^{u_i^+} - z^{u_i^-}y^{v_i}$ is a binomial of one of the three forms in~\Cref{eq:binomialMoves}, and for all $i\in [L-1]$, $p_i+u_i^{-}= p_{i+1}+ u_{i+1}^+$ and $q_i + v_i = q_{i+1}$.

Let $Q \in \mathbb{R}^{n \times n}$ be a matrix whose column $i$ is labeled by the family $\fa_\cg(i)$, and whose column labels agree with the sparsity of the corresponding column. We associate to $z_\cg$ the pair $(1, Q)$  where $1 \in \cc[y^\pm]$. In \Cref{eq: binomial in doubled chickering}, each successive partial sum adds an extra binomial that is a step in transforming $z_\cg$ into $z_\ch$.
By Remark~\ref{rem:unitaryTransformations} it remains to show that the binomials in Equation~\eqref{eq:binomialMoves} 
induce successive orthogonal transformations on $Q$ that preserve the property that column labels of the matrix obtained in each step agree (generically) with the sparsity of the matrix and the exponent vector of the current modification of $z_\cg$ under the partial sum.

The choice of transformation in step $i$ depends on whether the binomial $z^{u_i^+} - z^{u_i^-}y^{v_i}$ is of the first, second, or third form in~\Cref{eq:binomialMoves}. In particular, given a pair $(m, M)$ of monomial $m\in \mathbb{C}[y^\pm]$ and matrix $M \in \mathbb{R}^{n\times n}$, the first binomial acts by
\begin{align*}(m,M) \mapsto \left(m \cdot  \frac{y_{A \to c}y_{A \cup c \to b}}{y_{A \to b}y_{A \cup b \to c}}, M\right)\end{align*}
and column labels remain unchanged. Clearly the column labels still respect the sparsity since the matrix is unchanged. The second binomial acts by
\begin{align*}(m,M) \mapsto \left(m \cdot  \frac{y_{A \to c}}{y_{A \to b}}, M\right)\end{align*}
and the column of $M$ labeled $A \cup b \to c$ is relabeled $A \cup c \to b$. Since these two families have the same support, the column labeling still agrees with the sparsity. For the third binomial,
\begin{align*}(m,M) \mapsto (m, MU),\end{align*}
where $U$ is a proper Givens transformation given by \cite[Definition~6]{pmlr-v119-ghassami20a}. The sparsity claim is shown in \cite[Proposition~3]{pmlr-v119-ghassami20a}.

Consequently, the binomials that transform $z_\cg$ into $z_\ch$ induce a sequence of orthogonal transformations that map $(1, Q)$, where $Q$ has sparsity determined by $\cg$, to some $(1, Q^\prime)$, where $Q^\prime$ has sparsity determined by $\ch$. Hence, by \Cref{rem:unitaryTransformations}, $QQ^T = Q'{Q'}^T$ is in the Euclidean closure of the set of precision matrices of random vectors in $\cm(\ch)$.
\end{proof}

\begin{example}
    Consider the directed cycles $\cg$ and $\ch$ in \Cref{fig:cycle reversal example}. Then 
    \begin{align*}
        y_{\emptyset \to 1}\left(z_\cg - z_\ch\right) &= z_{2 \to 3}z_{3 \to 1}\left(y_{\emptyset \to 1}z_{1 \to 2} - y_{\emptyset \to 2}z_{2 \to 1} \right) \\
        &+ z_{3 \to 1}z_{2 \to 1} \left(y_{\emptyset \to 2}z_{2 \to 3} - y_{\emptyset \to 3}z_{3 \to 2} \right) \\
        &+ z_{2 \to 1}z_{3 \to 2}\left(y_{\emptyset \to 3}z_{3 \to 1} - y_{\emptyset \to 1} z_{1 \to 3} \right),
    \end{align*}
    and so $y_{\emptyset \to 1} \left(z_\cg - z_\ch \right) \in \chick_n^\prime$. Since $y_{\emptyset \to 1}$ is invertible, $z_\cg - z_\ch \in \chick_n^\prime$. Figure~\ref{fig:cycle reversal example} illustrates how each of these binomials represents a change on the matrix $Q$. In this figure, the boxes represent the distinguished child in each column's associated family.

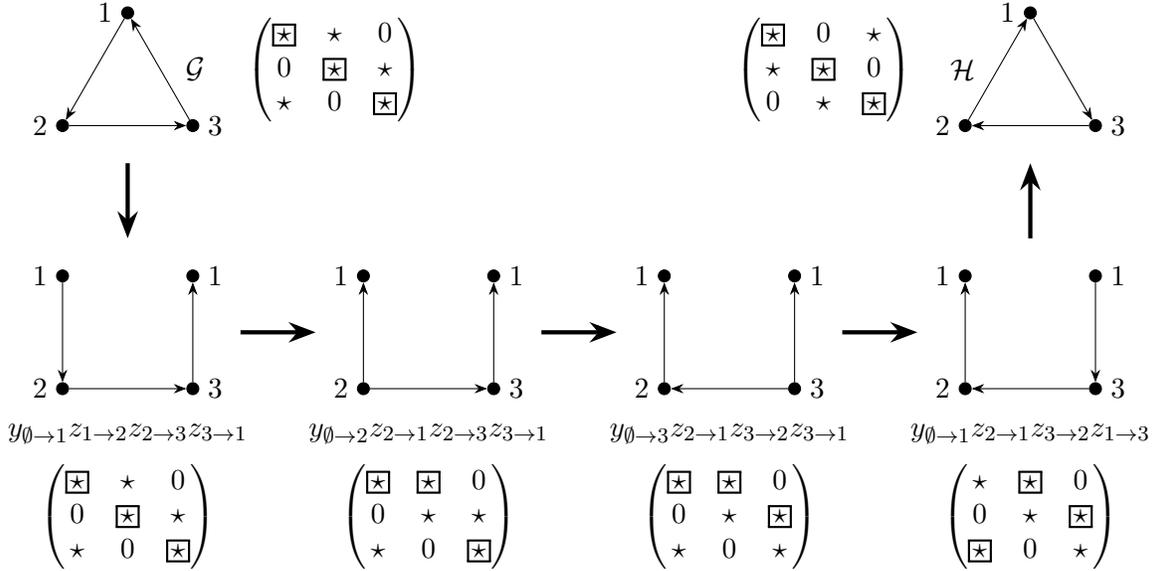
\begin{figure}
    \centering    
    \begin{tikzpicture}[>=Stealth, shorten >=0pt, shorten <=0pt, scale=1]

% Define side length of triangle
\def\side{1.5}
\def\sep{3} % separation between graphs
\pgfmathsetmacro{\h}{sqrt(3)/2} % height of equilateral triangle

% Graph 1 (G)
\begin{scope}[shift={(0,3.5)}]
\node at (0.9, 0.25) {$\mathcal{G}$}; 

\node[wB] (X1) at (0, 1) {};
\node[wB] (X2) at (-\h, -0.5) {};
\node[wB] (X3) at (\h, -0.5) {};

\node at (-0.3, 1) {1};
\node at (-\h - 0.3, -0.5) {2};
\node at (\h + 0.3, -0.5) {3};

\draw[->] (X1) -- (X2);
\draw[->] (X2) -- (X3);
\draw[->] (X3) -- (X1);

\node at (2.75,0.25) {$\begin{pmatrix}
    \boxedstar & \star &     0 \\
        0 & \boxedstar & \star \\
    \star &     0 & \boxedstar \\
\end{pmatrix}$};
\end{scope}

% Graph 2
\begin{scope}[shift={(0,0)}]
% Arrow between Graph 1 and Graph 2
\draw [ultra thick, -Stealth] (0, 2.5) -- (0,1.5);

\node[wB] (X1) at (-\h, 1) {};
\node[wB] (X2) at (-\h, -0.5) {};
\node[wB] (X3) at (\h, -0.5) {};
\node[wB] (X4) at (\h, 1) {};

\node at (-\h - 0.3, 1) {1};
\node at (-\h - 0.3, -0.5) {2};
\node at (\h + 0.3, -0.5) {3};
\node at (\h + 0.3, 1) {1};

\draw[->] (X1) -- (X2);
\draw[->] (X2) -- (X3);
\draw[->] (X3) -- (X4);

\node at (0,-1.1) {$y_{\emptyset \to 1}z_{1 \to 2}z_{2 \to 3}z_{3 \to 1}$};

\node at (0,-2.2) {$\begin{pmatrix}
    \boxedstar & \star &     0 \\
        0 & \boxedstar & \star \\
    \star &     0 & \boxedstar \\
\end{pmatrix}$};
\end{scope}

% Graph 3
\begin{scope}[shift={(4,0)}]
% Arrow between Graph 2 and Graph 3
\draw [ultra thick, -Stealth] (-2.5, 0.25) -- (-1.5,0.25);

\node[wB] (X1) at (-\h, 1) {};
\node[wB] (X2) at (-\h, -0.5) {};
\node[wB] (X3) at (\h, -0.5) {};
\node[wB] (X4) at (\h, 1) {};

\node at (-\h - 0.3, 1) {1};
\node at (-\h - 0.3, -0.5) {2};
\node at (\h + 0.3, -0.5) {3};
\node at (\h + 0.3, 1) {1};

\draw[->] (X2) -- (X1);
\draw[->] (X2) -- (X3);
\draw[->] (X3) -- (X4);

\node at (0,-1.1) {$y_{\emptyset \to 2}z_{2 \to 1}z_{2 \to 3}z_{3 \to 1}$};

\node at (0,-2.2) {$\begin{pmatrix}
    \boxedstar & \boxedstar &     0 \\
        0 & \star & \star \\
    \star &     0 & \boxedstar \\
\end{pmatrix}$};
\end{scope}

% Graph 4
\begin{scope}[shift={(8,0)}]
% Arrow between Graph 2 and Graph 3
\draw [ultra thick, -Stealth] (-2.5, 0.25) -- (-1.5,0.25);

\node[wB] (X1) at (-\h, 1) {};
\node[wB] (X2) at (-\h, -0.5) {};
\node[wB] (X3) at (\h, -0.5) {};
\node[wB] (X4) at (\h, 1) {};

\node at (-\h - 0.3, 1) {1};
\node at (-\h - 0.3, -0.5) {2};
\node at (\h + 0.3, -0.5) {3};
\node at (\h + 0.3, 1) {1};

\draw[->] (X2) -- (X1);
\draw[->] (X3) -- (X2);
\draw[->] (X3) -- (X4);

\node at (0,-1.1) {$y_{\emptyset \to 3}z_{2 \to 1}z_{3 \to 2}z_{3 \to 1}$};

\node at (0,-2.2) {$\begin{pmatrix}
    \boxedstar & \boxedstar &     0 \\
        0 & \star & \boxedstar \\
    \star &     0 & \star \\
\end{pmatrix}$};
\end{scope}

% Graph 5
\begin{scope}[shift={(12,0)}]
% Arrow between Graph 5 and Graph 6
\draw [ultra thick, -Stealth] (0, 1.5) -- (0,2.5);

% Arrow between Graph 4 and Graph 5
\draw [ultra thick, -Stealth] (-2.5, 0.25) -- (-1.5,0.25);

\node[wB] (X1) at (-\h, 1) {};
\node[wB] (X2) at (-\h, -0.5) {};
\node[wB] (X3) at (\h, -0.5) {};
\node[wB] (X4) at (\h, 1) {};

\node at (-\h - 0.3, 1) {1};
\node at (-\h - 0.3, -0.5) {2};
\node at (\h + 0.3, -0.5) {3};
\node at (\h + 0.3, 1) {1};

\draw[->] (X2) -- (X1);
\draw[->] (X3) -- (X2);
\draw[->] (X4) -- (X3);

\node at (0,-1.1) {$y_{\emptyset \to 1}z_{2 \to 1}z_{3 \to 2}z_{1 \to 3}$};

\node at (0,-2.2) {$\begin{pmatrix}
    \star & \boxedstar &     0 \\
        0 & \star & \boxedstar \\
    \boxedstar &     0 & \star \\
\end{pmatrix}$};
\end{scope}

% Graph 6 (H)
\begin{scope}[shift={(12,3.5)}]
\node at (-0.9, 0.25) {$\mathcal{H}$}; 

\node[wB] (X1) at (0, 1) {};
\node[wB] (X2) at (-\h, -0.5) {};
\node[wB] (X3) at (\h, -0.5) {};

\node at (-0.3, 1) {1};
\node at (-\h - 0.3, -0.5) {2};
\node at (\h + 0.3, -0.5) {3};

\draw[->] (X1) -- (X3);
\draw[->] (X2) -- (X1);
\draw[->] (X3) -- (X2);

\node at (-2.75,0.25) {$\begin{pmatrix}
    \boxedstar & 0 &     \star \\
    \star & \boxedstar & 0 \\
        0 &     \star & \boxedstar \\
\end{pmatrix}$};
\end{scope}

\end{tikzpicture}
    \caption{Two covariance equivalent graphs related by reversing a cycle. Boxed stars represent the distinguished child in the family that indexes the column.}
    \label{fig:cycle reversal example}
\end{figure}
\end{example}

\begin{remark}
    One might hope that imset equivalence and covariance equivalence are the same relations when restricted to graphs that have the same skeleton. This would be a partial converse to Theorem~\ref{thm: main}. However, we see in the following example that this is not the case. As a consequence, imset equivalence classes not only refine covariance equivalence classes, but also strictly refine the intersection of the covariance equivalence and skeletal equivalence relations.
\end{remark}

\begin{example}
\label{ex:converseFails}
Let $\cg$ be the graph on the left in Figure~\ref{fig:converseCounterexample} and $\ch$ be the graph on the right. Figure~\ref{fig:converseCounterexample} shows that $\cg$ and $\ch$ are related by first reversing the 135-cycle (in red) and then reversing the 246-cycle (in blue), where cycle reversal is done in the sense of~\cite{pmlr-v119-ghassami20a}. Hence $\cg$ and $\ch$ are covariance equivalent by \cite{pmlr-v119-ghassami20a}. However
\begin{align*}c_\cg(235) = 0 \neq 1 = c_\ch(235),\end{align*}
and so $\cg$ and $\ch$ are not imset equivalent. In short, imset equivalent graphs must have the same 3-cycles but covariance equivalent graphs do not (even if they have the same skeleton). Up to isomorphism, this is the unique pair of 2-cycle free graphs on at most 6 vertices which are covariance equivalent and not imset equivalent.
\begin{figure}
    \centering
\begin{tikzpicture}[scale=1.5]

\begin{scope}[shift={(0,0)}]
\node at (1.2,0) {\small 1 \normalsize};
\node at (0.6,1.039) {\small 2 \normalsize};
\node at (-0.6,1.039) {\small 3 \normalsize};
\node at (-1.2,0) {\small 4 \normalsize};
\node at (-0.6,-1.039) {\small 5 \normalsize};
\node at (0.6,-1.039) {\small 6 \normalsize};

\draw[-Stealth, shorten >=3pt,black] (1,0) to (0.5,0.866);
\draw[-Stealth, shorten >=3pt,black] (0.5,0.866) to (-0.5,0.866);
\draw[-Stealth, shorten >=3pt,black] (-0.5,0.866) to (-1,0);
\draw[-Stealth, shorten >=3pt,black] (1,0) to (0.5,-0.866);
\draw[-Stealth, shorten >=3pt,black] (0.5,-0.866) to (-0.5,-0.866);
\draw[-Stealth, shorten >=3pt,black] (-0.5,-0.866) to (-1,0);

\draw[-Stealth, shorten >=3pt,blue,thick] (0.5,0.866)--(-1,0);
\draw[-Stealth, shorten >=3pt,blue,thick] (-1,0)--(0.5,-0.866);
\draw[-Stealth, shorten >=3pt,blue,thick] (0.5,-0.866)--(0.5,0.866);

\draw[-Stealth, shorten >=3pt,red,thick] (1,0)--(-0.5,0.866);
\draw[-Stealth, shorten >=3pt,red,thick] (-0.5,0.866)--(-0.5,-0.866);
\draw[-Stealth, shorten >=3pt,red,thick] (-0.5,-0.866)--(1,0);

\draw[-Stealth, shorten >=3pt,black] (-1,0)--(1,0);
\draw[-Stealth, shorten >=3pt,black] (-0.5,0.866)--(0.5,-0.866);
\draw[-Stealth, shorten >=3pt,black] (-0.5,-0.866)--(0.5,0.866);

\node[wB] at (1,0) {};
\node[wB] at (0.5,0.866) {};
\node[wB] at (-0.5,0.866) {};
\node[wB] at (-1,0) {};
\node[wB] at (-0.5,-0.866) {};
\node[wB] at (0.5,-0.866) {};
\end{scope}

\draw[-Stealth,ultra thick] (1.5,0)--(2.5,0);

\begin{scope}[shift={(4,0)}]
\node at (1.2,0) {\small 1 \normalsize};
\node at (0.6,1.039) {\small 2 \normalsize};
\node at (-0.6,1.039) {\small 3 \normalsize};
\node at (-1.2,0) {\small 4 \normalsize};
\node at (-0.6,-1.039) {\small 5 \normalsize};
\node at (0.6,-1.039) {\small 6 \normalsize};

\draw[-Stealth, bend left=10, shorten >=3pt,black] (1,0) to (0.5,0.866);
\draw[-Stealth, bend left=10, shorten >=3pt,black] (0.5,0.866) to (1,0);

\draw[-Stealth, bend left=10, shorten >=3pt,black] (0.5,-0.866) to (-0.5,0.866);
\draw[-Stealth, bend left=10, shorten >=3pt,black] (-0.5,0.866) to (0.5,-0.866);

\draw[-Stealth, bend left=10, shorten >=3pt,black] (-0.5,-0.866) to (-1,0);
\draw[-Stealth, bend left=10, shorten >=3pt,black] (-1,0) to (-0.5,-0.866);

\draw[-Stealth, shorten >=3pt,blue,thick] (0.5,0.866)--(-1,0);
\draw[-Stealth, shorten >=3pt,blue,thick] (-1,0)--(0.5,-0.866);
\draw[-Stealth, shorten >=3pt,blue,thick] (0.5,-0.866)--(0.5,0.866);

\draw[-Stealth, shorten >=3pt,red,thick] (-0.5,0.866)--(1,0);
\draw[-Stealth, shorten >=3pt,red,thick] (-0.5,-0.866)--(-0.5,0.866);
\draw[-Stealth, shorten >=3pt,red,thick] (1,0)--(-0.5,-0.866);

\draw[-Stealth, shorten >=3pt,black] (-0.5,0.866)--(-1,0);
\draw[-Stealth, shorten >=3pt,black] (1,0)--(0.5,-0.866);
\draw[-Stealth, shorten >=3pt,black] (-0.5,-0.866)--(0.5,0.866);

\node[wB] at (1,0) {};
\node[wB] at (0.5,0.866) {};
\node[wB] at (-0.5,0.866) {};
\node[wB] at (-1,0) {};
\node[wB] at (-0.5,-0.866) {};
\node[wB] at (0.5,-0.866) {};
\end{scope}

\draw[-Stealth,ultra thick] (5.5,0)--(6.5,0);

\begin{scope}[shift={(8,0)}]
\node at (1.2,0) {\small 1 \normalsize};
\node at (0.6,1.039) {\small 2 \normalsize};
\node at (-0.6,1.039) {\small 3 \normalsize};
\node at (-1.2,0) {\small 4 \normalsize};
\node at (-0.6,-1.039) {\small 5 \normalsize};
\node at (0.6,-1.039) {\small 6 \normalsize};

\draw[-Stealth, shorten >=3pt,black] (0.5,0.866) to (1,0);
\draw[-Stealth, shorten >=3pt,black] (-0.5,0.866) to (0.5,0.866);
\draw[-Stealth, shorten >=3pt,black] (-0.5,0.866) to (-1,0);
\draw[-Stealth, shorten >=3pt,black] (1,0) to (0.5,-0.866);
\draw[-Stealth, shorten >=3pt,black] (-0.5,-0.866) to (0.5,-0.866);
\draw[-Stealth, shorten >=3pt,black] (-1,0)--(-0.5,-0.866);

\draw[-Stealth, shorten >=3pt,blue,thick] (-1,0)--(0.5,0.866);
\draw[-Stealth, shorten >=3pt,blue,thick] (0.5,-0.866)--(-1,0);
\draw[-Stealth, shorten >=3pt,blue,thick] (0.5,0.866)--(0.5,-0.866);

\draw[-Stealth, shorten >=3pt,red,thick] (-0.5,0.866)--(1,0);
\draw[-Stealth, shorten >=3pt,red,thick] (-0.5,-0.866)--(-0.5,0.866);
\draw[-Stealth, shorten >=3pt,red,thick] (1,0)--(-0.5,-0.866);

\draw[-Stealth, shorten >=3pt,black] (1,0)--(-1,0);
\draw[-Stealth, shorten >=3pt,black] (0.5,-0.866)--(-0.5,0.866);
\draw[-Stealth, shorten >=3pt,black] (-0.5,-0.866)--(0.5,0.866);

\node[wB] at (1,0) {};
\node[wB] at (0.5,0.866) {};
\node[wB] at (-0.5,0.866) {};
\node[wB] at (-1,0) {};
\node[wB] at (-0.5,-0.866) {};
\node[wB] at (0.5,-0.866) {};
\end{scope}

\end{tikzpicture}
    \caption{Two covariance equivalent, 2-cycle free graphs with different characteristic imset vectors.}
    \label{fig:converseCounterexample}
\end{figure}
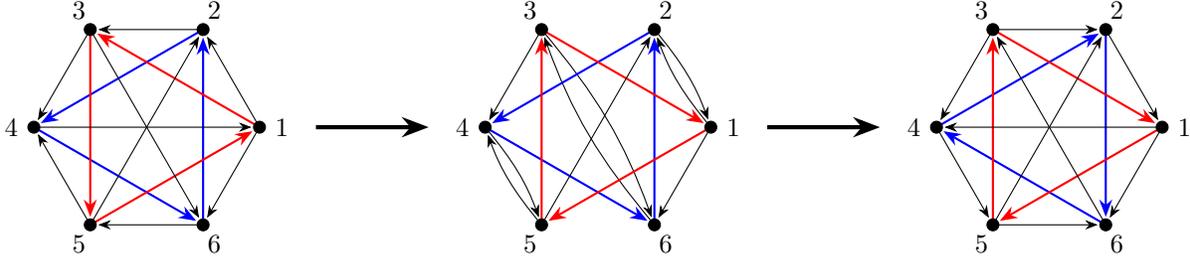
\end{example}

\section{Future Directions}
\label{sec:futureDirections}

In this section we discuss some potential directions for future study.

\subsection{Causal Discovery Algorithm} Theorem~\ref{thm: main} states that two graphs with the same characteristic (or standard) imset vector are covariance equivalent, and therefore, model equivalent when the noise is Gaussian. As a consequence, greedy search approaches to causal discovery can search over the set of standard imset vectors instead of the set of all directed (possibly cyclic) graphs. This could yield faster search algorithms as it avoids redundant scoring of multiple graphs within the same imset equivalence class. See~\cite{pmlr-v119-ghassami20a, pmlr-v124-amendola20a} for scoring methods and search algorithms in the cyclic and Gaussian settings. Notably, to search over imsets instead of graphs, one needs a way to recover a graph in the fiber of a standard imset vector.

\begin{question*}
Given a vector $s \in \mathbb{Z}^{2^n}$ with at most $2n$ nonzero coordinates, decide in polynomial time if there is a graph $\cg$ such that $s = \simv_\cg$. How can one recover $\cg$ (or any other graph in the associated fiber) from $\simv_\cg$? 
\end{question*}

Difficulties arise from searching over characteristic imset vectors since they require exponential space in the number of random variables. However, standard imset vectors can be stored in far less space since at most $2n$ of the $2^n$ coordinates are nonzero for a graph on $n$ vertices. Additionally, since graphs in the same imset equivalence class have the same skeleton, it would be interesting to investigate notions of patterns and essential graphs, which could address issues of space in a way that aligns more with the traditional literature.

\subsection{Polyhedral Structure of Imsets for Cyclic Graphs}

Score-based approaches to causal discovery often require a set of combinatorial moves that transform one graph or model equivalence class into another. For instance, in Greedy Equivalence Search (GES) the moves required to learn an acyclic graph are edge additions, edge deletions, and edge reversals. Such moves can be realized as edges of the characteristic imset polytope \cite{xi2015cim,lrs2022edges,lrs2023geometric,linusson2023rhombus}, which is the convex hull of all of the characteristic imsets corresponding to DAGs on a specified number of vertices.

\begin{question*}
What are the edges of the cyclic characteristic imset polytope?
\end{question*}

Studying the edges of the characteristic imset polytope can provide additional moves in greedy search algorithms that prevent the algorithm from terminating in a local maximum of a score function instead of the global maximum \cite{lrs2023geometric}.

\subsection{Interventional Experiments} Interventional characteristic imsets for acyclic graphs are introduced in~\cite{hollering2024hyperplanerepresentationsinterventionalcharacteristic}. 
Interventional distributions refer to the distributions of an SEM after some of its defining equations are modified. Such interventions can reveal additional information about the underlying causal graph and learn a more refined model equivalence class.

\begin{question*}
Are cyclic graphs with the same interventional characteristic imsets in the same model equivalence class?
\end{question*}

\subsection{Properties of the Chickering Ideal} The Chickering ideal is a toric ideal that captures complete information about covariance equivalence for acyclic graphs and partial information in the cyclic case. Since the Chickering ideal is closely related to foundational results in causal discovery such as the transformational characterization of covariance equivalence, it would be interesting to better understand what algebraic properties this ideal and the associated Chickering variety have.

\begin{question*}
What is the degree of the Chickering variety? Is it smooth? Is it normal?
\end{question*}

Since a Gr\"obner basis for the Chickering ideal seems intractable, it would be interesting to see arguments that address these questions.

\subsection{Model Equivalence and Characteristic Imset Vectors} 
\Cref{ex:converseFails} shows that imset equivalence is a strictly stronger requirement than covariance equivalence for graphs with the same skeleton. When cycles are allowed in the causal graph, model equivalence is also a strictly stronger requirement than covariance equivalence~\cite{Lacerda2008}. 
\begin{question*}
    What is the relationship between model equivalence and imset equivalence? Do model equivalent graphs with the same skeleton share the same characteristic imset vector?
\end{question*}

\section*{Acknowledgements}

We thank Liam Solus and Mathias Drton for helpful discussions. Part of this research was performed while the authors were visiting the Institute for Mathematical and Statistical Innovation (IMSI), which is supported by the National Science Foundation (Grant No. DMS-1929348). Pardis Semnani was supported by a Vanier Canada Graduate Scholarship, a Canada CIFAR AI Chair grant awarded to Elina Robeva, and an NSERC Discovery Grant (DGECR-2020-00338).

\bibliography{citations.bib}{}

\begin{thebibliography}{10}

\bibitem{pmlr-v124-amendola20a}
Carlos Amendola, Philipp Dettling, Mathias Drton, Federica Onori, and Jun Wu.
\newblock Structure learning for cyclic linear causal models.
\newblock In Jonas Peters and David Sontag, editors, {\em Proceedings of the 36th Conference on Uncertainty in Artificial Intelligence (UAI)}, volume 124 of {\em Proceedings of Machine Learning Research}, pages 999--1008. PMLR, 03--06 Aug 2020.

\bibitem{andersson1997markov}
Steen~A. Andersson, David Madigan, and Michael~D. Perlman.
\newblock A characterization of {M}arkov equivalence classes for acyclic digraphs.
\newblock {\em Annals of Statistics}, 25(2):505--541, 1997.

\bibitem{Bernstein2020}
Daniel Bernstein, Basil Saeed, Chandler Squires, and Caroline Uhler.
\newblock Ordering-based causal structure learning in the presence of latent variables.
\newblock In {\em Proceedings of the Twenty Third International Conference on Artificial Intelligence and Statistics}, volume 108, pages 4098--4108. PMLR, 2020.

\bibitem{chickering1995transformational}
David~Maxwell Chickering.
\newblock A transformational characterization of equivalent {B}ayesian network structures.
\newblock In {\em Uncertainty in artificial intelligence ({M}ontreal, {PQ}, 1995)}, pages 87--98. Morgan Kaufmann, San Francisco, CA, 1995.

\bibitem{Chickering2002OptimalSI}
David~Maxwell Chickering.
\newblock Optimal structure identification with greedy search.
\newblock {\em J. Mach. Learn. Res.}, 3:507--554, 2002.

\bibitem{cussens2013maximum}
James Cussens, Mark Bartlett, Elinor~M Jones, and Nuala~A Sheehan.
\newblock Maximum likelihood pedigree reconstruction using integer linear programming.
\newblock {\em Genetic epidemiology}, 37(1):69--83, 2013.

\bibitem{cussens2017scoreequiv}
James Cussens, David Haws, and Milan Studen\'{y}.
\newblock Polyhedral aspects of score equivalence in {B}ayesian network structure learning.
\newblock {\em Math. Program.}, 164(1-2):285--324, 2017.

\bibitem{ijcai2017p708}
James Cussens, Matti Järvisalo, Janne~H. Korhonen, and Mark Bartlett.
\newblock Bayesian network structure learning with integer programming: Polytopes, facets and complexity (extended abstract).
\newblock In {\em Proceedings of the Twenty-Sixth International Joint Conference on Artificial Intelligence, {IJCAI-17}}, pages 4990--4994, 2017.

\bibitem{eisenbudsturmfels}
David Eisenbud and Bernd Sturmfels.
\newblock Binomial ideals.
\newblock {\em Duke Math. J.}, 84(1):1--45, 1996.

\bibitem{forré2017markov}
Patrick Forr{\'e} and Joris~M Mooij.
\newblock {M}arkov properties for graphical models with cycles and latent variables.
\newblock {\em arXiv:1710.08775}, 2017.

\bibitem{Friedman2000}
N.~Friedman, M~Linial, I.~Nachman, and D.~Peter.
\newblock Using {B}ayesian networks to analyze expression data.
\newblock {\em Journal of Computational Biology}, 7(3-4):601–620, 2000.

\bibitem{pmlr-v119-ghassami20a}
Amiremad Ghassami, Alan Yang, Negar Kiyavash, and Kun Zhang.
\newblock Characterizing distribution equivalence and structure learning for cyclic and acyclic directed graphs.
\newblock In Hal~Daumé III and Aarti Singh, editors, {\em Proceedings of the 37th International Conference on Machine Learning}, volume 119 of {\em Proceedings of Machine Learning Research}, pages 3494--3504. PMLR, 13--18 Jul 2020.

\bibitem{Goldberger1972}
Arthur~S. Goldberger.
\newblock Structural equation methods in the social sciences.
\newblock {\em Econometrica}, 40(6):979--1001, 1972.

\bibitem{Haavelmo1943}
Trygve Haavelmo.
\newblock The statistical implications of a system of simultaneous equations.
\newblock {\em Econometrica}, 11(1):1--12, 1943.

\bibitem{Heckerman1994LearningBN}
David~E. Heckerman, Dan Geiger, and David~Maxwell Chickering.
\newblock Learning {B}ayesian networks: The combination of knowledge and statistical data.
\newblock {\em Machine Learning}, 20:197--243, 1995.

\bibitem{hollering2024hyperplanerepresentationsinterventionalcharacteristic}
Benjamin Hollering, Joseph Johnson, and Liam Solus.
\newblock Hyperplane representations of interventional characteristic imset polytopes, 2024.

\bibitem{jaakkola2010learning}
Tommi Jaakkola, David Sontag, Amir Globerson, and Marina Meila.
\newblock Learning {B}ayesian network structure using {LP} relaxations.
\newblock In {\em Proceedings of the thirteenth international conference on artificial intelligence and statistics}, pages 358--365. JMLR Workshop and Conference Proceedings, 2010.

\bibitem{Lacerda2008}
Gustavo Lacerda, Peter Spirtes, Joseph Ramsey, and Patrik~O. Hoyer.
\newblock Discovering cyclic causal models by independent components analysis.
\newblock In {\em Proceedings of the Twenty-Fourth Conference on Uncertainty in Artificial Intelligence}, UAI'08, page 366–374, Arlington, Virginia, USA, 2008. AUAI Press.

\bibitem{linusson2023rhombus}
Svante Linusson and Petter Restadh.
\newblock Rhombus criterion and the chordal graph polytope.
\newblock 2023.

\bibitem{lrs2022edges}
Svante Linusson, Petter Restadh, and Liam Solus.
\newblock On the edges of characteristic imset polytopes.
\newblock {\em arXiv preprint arXiv:2209.07579}, 2022.

\bibitem{lrs2023geometric}
Svante Linusson, Petter Restadh, and Liam Solus.
\newblock Greedy causal discovery is geometric.
\newblock {\em SIAM J. Discrete Math.}, 37(1):233--252, 2023.

\bibitem{Mason1953}
Samuel~J. Mason.
\newblock Feedback theory-some properties of signal flow graphs.
\newblock {\em Proceedings of the IRE}, 41(9):1144--1156, 1953.

\bibitem{Mason1956}
Samuel~J. Mason.
\newblock Feedback theory-further properties of signal flow graphs.
\newblock {\em Proceedings of the IRE}, 44(7):920--926, 1956.

\bibitem{meek1997graphical}
Christopher Meek.
\newblock {\em Graphical Models: Selecting causal and statistical models}.
\newblock Phd thesis, Carnegie Mellon University, 1997.

\bibitem{Pearl2009}
Judea Pearl.
\newblock {\em Causality}.
\newblock Cambridge University Press, 2 edition, 2009.

\bibitem{pearlcausalitymodels}
Judea Pearl.
\newblock {\em Causality: Models, Reasoning and Inference}.
\newblock Cambridge University Press, USA, 2nd edition, 2009.

\bibitem{Raskutti2018}
Garvesh Raskutti and Caroline Uhler.
\newblock Learning directed acyclic graph models based on sparsest permutations.
\newblock {\em Stat}, 7(1):e183, 2018.
\newblock e183 sta4.183.

\bibitem{richardson2013discovery}
Thomas Richardson.
\newblock A discovery algorithm for directed cyclic graphs.
\newblock In {\em Proceedings of the Twelfth International Conference on Uncertainty in Artificial Intelligence}, UAI'96, page 454–461, San Francisco, CA, USA, 1996. Morgan Kaufmann Publishers Inc.

\bibitem{Robins2000}
J.~M. Robins, M.~A. Hernán, and B.~Brumback.
\newblock Marginal structural models and causal inference in epidemiology.
\newblock {\em Epidemiology}, 11(5):550--560, 2000.

\bibitem{SemnaniRobeva+2025}
Pardis Semnani and Elina Robeva.
\newblock Causal structure learning in directed, possibly cyclic, graphical models.
\newblock {\em Journal of Causal Inference}, 13(1):20240037, 2025.

\bibitem{Solus2021}
L~Solus, Y~Wang, and C~Uhler.
\newblock Consistency guarantees for greedy permutation-based causal inference algorithms.
\newblock {\em Biometrika}, 108(4):795--814, 01 2021.

\bibitem{SpirtesGlymourClark1993}
Peter Spirtes, Clark Glymour, and Richard Scheines.
\newblock {\em Causation, Prediction, and Search}, volume~81.
\newblock Springer New York, 01 1993.

\bibitem{studeny2005book}
Milan Studen\'{y}.
\newblock {\em Probabilistic conditional independence structures}.
\newblock Information Science and Statistics. Springer, London, 2005.

\bibitem{studeny2010characteristic}
Milan Studen{\`y}, Raymond Hemmecke, and Silvia Lindner.
\newblock Characteristic imset: a simple algebraic representative of a {B}ayesian network structure.
\newblock In {\em Proceedings of the 5th European workshop on probabilistic graphical models}, pages 257--264. HIIT Publications, 2010.

\bibitem{sullivant2018algstat}
Seth Sullivant.
\newblock {\em Algebraic statistics}, volume 194 of {\em Graduate Studies in Mathematics}.
\newblock American Mathematical Society, Providence, RI, 2018.

\bibitem{Teyssier2005OrderingBasedSA}
Marc Teyssier and Daphne Koller.
\newblock Ordering-based search: A simple and effective algorithm for learning {B}ayesian networks.
\newblock In {\em Proceedings of the Twenty-First Conference on Uncertainty in Artificial Intelligence}, UAI'05, page 584–590, Arlington, Virginia, USA, 2005. AUAI Press.

\bibitem{tsamardinos2006max}
Ioannis Tsamardinos, Laura~E Brown, and Constantin~F Aliferis.
\newblock The max-min hill-climbing {B}ayesian network structure learning algorithm.
\newblock {\em Machine learning}, 65:31--78, 2006.

\bibitem{verma2022equivalence}
Thomas~S Verma and Judea Pearl.
\newblock Equivalence and synthesis of causal models.
\newblock In {\em Probabilistic and Causal Inference: The Works of Judea Pearl}, pages 221--236. 2022.

\bibitem{xi2015cim}
Jing Xi and Ruriko Yoshida.
\newblock The characteristic imset polytope of {B}ayesian networks with ordered nodes.
\newblock {\em SIAM J. Discrete Math.}, 29(2):697--715, 2015.

\end{thebibliography}
\bibliographystyle{plain}
\end{document}